\documentclass[a4paper,11pt,reqno]{amsart}
\usepackage{amssymb}
\usepackage{amscd}
\usepackage{amsthm}
\usepackage{amsmath}
\usepackage[all]{xy}
\usepackage{extarrows}
\usepackage{rotating}
\usepackage{tikz-cd}
\usepackage{leftidx}
\usepackage{bm}
\usepackage{dynkin-diagrams}
\usepackage{mathrsfs}
\usepackage{rotating}
\usepackage{upgreek}
\numberwithin{equation}{section}

\usepackage[pagebackref=true]{hyperref}
\renewcommand*{\backref}[1]{}
\renewcommand*{\backrefalt}[4]{[{\tiny%
		\ifcase #1 Not cited.%
		\or Cited on page~#2.%
		\else Cited on pages #2.%
		\fi%
	}]}

\newcommand{\fd}{\mathsf{fd}}

\renewenvironment{proof}{{ \textbf{Proof}.}}{\qed}
\newtheorem{Thm}{Theorem}[section]
\newtheorem{Lem}[Thm]{Lemma}
\newtheorem{Def}[Thm]{Definition}
\newtheorem{Cor}[Thm]{Corollary}
\newtheorem{Prop}[Thm]{Proposition}
\newtheorem{Ex1}[Thm]{Example}
\newtheorem{Rem1}[Thm]{Remark}

\newcommand{\cone}{\mathrm{Cone}}
\newcommand{\Hom}{\mathrm{Hom}}
\newcommand{\thick}{\mathrm{thick}\,}
\newcommand{\add}{\mathrm{add}}
\newcommand{\per}{\mathrm{per}}
\newcommand{\Ext}{\mathrm{Ext}}
\newcommand{\End}{\mathrm{End}}
\newcommand{\obj}{\mathrm{obj}}

\newcommand{\bijar}[1][]{%
	\ar[#1]
	\ar@<0.7ex>@{}[#1]|-*[@]{\sim}}

\newcommand{\pvd}{\mathrm{pvd}}

\newenvironment{Rem}{\begin{Rem1}\rm}{\end{Rem1}}
\newenvironment{Ex}{\begin{Ex1}\rm}{\end{Ex1}}

\usepackage{geometry}
\newcommand{\Si}{\Sigma}

\newcommand{\ra}{\rightarrow}

\newcommand{\iso}{\xrightarrow{_\sim}}

\newcommand{\id}{\mathbf{1}}

\newcommand{\ca}{{\mathcal A}}

\newcommand{\cc}{{\mathcal C}}
\newcommand{\cd}{{\mathcal D}}

\newcommand{\cf}{{\mathcal F}}

\newcommand{\ch}{{\mathcal H}}

\newcommand{\cm}{{\mathcal M}}

\newcommand{\co}{{\mathcal O}}
\newcommand{\cp}{{\mathcal P}}

\newcommand{\cq}{{\mathcal Q}}
\newcommand{\cs}{{\mathcal S}}
\newcommand{\ct}{{\mathcal T}}
\newcommand{\cu}{{\mathcal U}}
\newcommand{\cv}{{\mathcal V}}

\newcommand{\cz}{{\mathcal Z}}

\begin{document}
	
	\title{Silting reduction, relative AGK's construction and Higgs construction}
	
	

	\author{Yilin Wu}
	\address{
University of Luxembourg\\                           
6, avenue de la Fonte, L-4364\\
Esch-sur-Alzette\\            Luxembourg
}
\email{yilin.wu@uni.lu}
	
	
	\dedicatory{}
	
	
	\begin{abstract}

We introduce the notion of a Calabi--Yau quadruple as a generalization of Iyama--Yang's Calabi--Yau triple. 
For each \((d+1)\)-Calabi--Yau quadruple, we show that the associated Higgs category is a \(d\)-Calabi--Yau Frobenius extriangulated category, which moreover admits a canonical \(d\)-cluster-tilting subcategory. Concrete examples arise from the construction of relative cluster categories and Higgs categories in the setting of ice quivers with potentials, as well as from the singularity category of an isolated singularity.  As an application, we prove that both the relative Amiot--Guo--Keller's construction and the Higgs construction of a \((d+1)\)-Calabi--Yau quadruple take silting reduction to Calabi--Yau reduction.

	\end{abstract}
	
	\maketitle
	\tableofcontents
	
	\section{Introduction}
	Triangulated and derived categories appear ubiquitously throughout mathematics, with applications in representation theory, algebraic geometry, algebraic topology, and mathematical physics. An important tool in the study of derived categories is provided by silting reduction and Calabi–Yau reduction \cite{IyamaYoshino2008,IY2018}. The concept of silting appeared for the first time in a paper of Keller--Vossieck\cite{KV1988}. Since then, silting theory and silting reduction have played significant roles in the study of triangulated categories, in particular derived categories \cite{IY2018}. Moreover, silting reduction has been shown to be closely related to Calabi–Yau reduction and has found wide applications in representation theory. In \cite{IY2018}, they introduced the notion of a Calabi--Yau triple and showed that Amiot--Guo--Keller’s construction (\cite{Am2008,Guo2011}) is a direct passage from tilting theory to cluster tilting theory.

	The aim of this paper is to introduce the notion of a Calabi--Yau quadruple as a generalization of Iyama--Yang's Calabi--Yau triple \cite{IY2018}. For each \((d+1)\)-Calabi--Yau quadruple, we show that the associated Higgs category is a Frobenius \(d\)-Calabi--Yau extriangulated category, which moreover admits a canonical \(d\)-cluster-tilting subcategory. 
	Let $\ct$ be a triangulated category, $\cm$ a subcategory of $\ct$, $\cp$ a subcategory of $\cm$ and $\ct^{\fd}\subseteq\ct$ a triangulated subcategory such that $(\ct,\ct^{\fd},\cm,\cp)$ is a $(d + 1)$-Calabi--Yau quadruple (see Definition \ref{Def:Calabi--Yau quadruple}). In particular, if we take $\cp$ to be the zero category, then $(\ct,\ct^{\fd},\cm)$ is a Calabi--Yau triple in the sense of Iyama--Yang.

	Applying relative Amiot--Guo--Keller’s construction, we obtain a triangulated category $\cc=\ct/\ct^{\fd}$ (AGK's relative cluster category) in which $\cp$ becomes a silting subcategory. Then the corresponding Higgs category $\ch$ is defined as
	$$\ch=(\cp[<0])^{\perp_{\cc}}\cap^{\perp_{\cc}}\!(\cp[>0]).$$ We show that the Higgs category $\ch$ is a $d$-Calabi--Yau Frobenius extriangulated category with projective-injective objects $\cp$ and $\cm$ is $d$-cluster-tilting subcategory of $\ch$, see Theorem \ref{Thm:Higgs is Fro}. And the stable category $\underline{\ch}$ is equivalent to the AGK's cluster category $\cu/\cu^{\fd}$, where $\cu=\ct/\thick(\cp)$ and $\cu^{\fd}=\ct^{\fd}\cap\thick(\cp)^{\perp}$. This generalize the results of \cite{Wu2023,KW2307} in the setting of ice quivers with potentials. Examples also arise from the singularity category of an isolated singularity, see Example \ref{Ex:isolated sing} and Example \ref{Ex:isolated singcat}.
	
	Let $\cq$ be a functorially finite subcategory of $\cm$. Then $\cq$ becomes a $d$-rigid subcategory of $\cc$. On the one hand, let $\ch'_{\cq}$ be the following extension closed subcategory of $\ch$
	$$\ch'_{\cq}=\{M\in\ch\,|\,\Hom_{\cc}(M,\cq[1]*\cq[2]*\cdots\cq[d-1])=0\}\subseteq\ch.$$ Then $\ch'_{\cq}$ is a Frobenius extriangulated category with projective-injective objects $\cp\cup\cq$, see Proposition \ref{Prop:H' is Fro}. The Calabi--Yau reduction of $\ch$ with respect to $\cq$ is defined as the additive quotient
	$\frac{\ch'_{\cq}}{[\cq]},$
	which is a $d$-Calabi--Yau Frobenius Extriangulated category and in which $\cm$ becomes a $d$-cluster-tilting subcategory.

	On the other hand, we first form the silting reduction of $(\ct,\ct^{\fd},\cm,\cp)$ with respect to $\cq$. We obtain a new $(d+1)$-Calabi--Yau quadruple $(\cv,\cv^{\fd},\cm,\cp)$, where $\cv=\ct/\thick(\cq)$ and $\cv^{\fd}=\ct^{\fd}\cap\thick(\cq)^{\perp}$. The corresponding relative cluster category and Higgs category are denote by $\cc_{\cq}$ and $\ch_{\cq}$ receptively. We prove that the two resulting $d$-Calabi--Yau Frobenius extriangulated categories $\ch_{\cq}$ and $\ch'_{\cq}/[\cq]$ are equivalent as Frobenius extriangulated categories, see Theorem \ref{Thm:CY reduction of Higgs}. Accordingly, Higgs construction takes silting reduction to Calabi--Yau reduction. The operations can be illustrated in the commutative diagram below.
	\[
	\SelectTips{cm}{10}
	\begin{xy}
		0;<1.1pt,0pt>:<0pt,-1.1pt>::
		(80,0) *+{(\ct,\ct^{\fd},\cm,\cp)} ="0", (0,50) *+{(\cv,\cv^{\fd},\cm,\cp)} ="1", (160,50) *+{\ch} ="2",
		(65,100) *+{\ch_{\cq}} ="3", (85,100) *+{\simeq} ="4", (100,100) *+{\frac{\ch'_{\cq}}{[\cq]}} ="5",
		(145,11) *+{\text{Higgs }}="6",
		(153,22) *+{\text{construction}}="7",
		(8,15) *+{\text{silting}}="8",
		(8,25) *+{\text{reduction}}="9",
		(5,75) *+{\text{Higgs}}="10",
		(7,85) *+{\text{construction}}="11",
		(163,75) *+{\text{Calabi--Yau }}="12",
		(163,85) *+{\text{reduction }}="13",
		(220,0) *+{},
		"0", {\ar@{~>}, "1"}, {\ar@{~>}, "2"},
		"1", {\ar@{~>}, "3"}, 
		"2", {\ar@{~>}, "5"},
	\end{xy}
	\]
	
In the following, we assume that $\ct$ is an algebraic triangulated category; that is, there exists a pretriangulated dg category $\ct_{dg}$ with $H^{0}(\ct_{dg}) \simeq \ct$. In this setting, the AGK's relative cluster category and the Higgs category also admit dg enhancements, which we denote by $\cc_{dg}$ and $\ch_{dg}$, respectively.  

Let $\ch'_{\cq,dg}$ be the full dg subcategory of $\cc_{dg}$ consisting of
objects in $\ch'_{\cq}$. It is an extension closed subcategory of $\cv_{dg}$.	Then we show that there is an exact quasi-isomorphism of exact dg categories (see Theorem \ref{Thm:dg exact quasi-isomorphism of H'})
$$\frac{\tau_{\leqslant0}\ch'_{\cq,dg}}{\cq_{dg}}\simeq\tau_{\leqslant0}\ch_{\cq,dg},$$
where the left hand-side quotient is the Drinfeld dg quotient and $\tau_{\leqslant0}$ is the mild truncation functor. In \cite{Chen2306}, Xiaofa Chen introduced the notion of the derived category of a dg exact category. After applying the construction of derived category to the above exact quasi-isomorphism, we obtain the following quasi-isomorphism of pretriangulated dg categories (see Theorem \ref{Thm:reduction of cc})
$$\cd^{b}_{dg}(\ch'_{\cq,dg})/\thick_{dg}(\cq)\simeq\cc_{\cq,dg}.$$
In particular, we have an equivalence of triangulated categories
$$\cd^{b}(\ch'_{\cq,dg})/\thick(\cq)\simeq\cc_{\cq}=\cv/\cv^{\fd}.$$

This can be illustrated by the following commutative diagram of operations.
	
\[
\SelectTips{cm}{10}
\begin{xy}
	0;<1.1pt,0pt>:<0pt,-1.1pt>::
	(80,0) *+{(\ct,\ct^{\fd},\cm,\cp)} ="0", (0,50) *+{(\cv,\cv^{\fd},\cm,\cp)} ="1", (160,50) *+{\cc=\ct/\ct^{\fd}} ="2",
	(65,100) *+{\cv/\cv^{\fd}} ="3", (85,100) *+{\simeq} ="4", (110,100) *+{\frac{\cd^{b}(\ch'_{\cq,dg})}{\thick(\cq)}} ="5",
	(145,11) *+{\text{Relative AGK's}}="6",
	(153,22) *+{\text{construction}}="7",
	(8,15) *+{\text{silting}}="8",
	(8,25) *+{\text{reduction}}="9",
	(-5,75) *+{\text{Relative AGK's}}="10",
	(-4,85) *+{\text{construction}}="11",
	(165,75) *+{\text{Calabi--Yau }}="12",
	(163,85) *+{\text{reduction}}="13",
	(220,0) *+{},
	"0", {\ar@{~>}, "1"}, {\ar@{~>}, "2"},
	"1", {\ar@{~>}, "3"}, 
	"2", {\ar@{~>}, "5"},
\end{xy}
\]

	\section*{Acknowledgments}
	The author would like to thank Xiaofa Chen for many interesting discussions and useful comments. He is very grateful to Martin Kalck, Junyang Liu, Matthew Pressland and Dong Yang for pointing out useful references and useful comments.
	
\section{Preliminaries}\label{Preliminaries}

Let $ \mathcal{T} $ be a triangulated category, and denote by $[1]$ the shift functor. We say that $\ct$ is \emph{idempotent complete} if any idempotent morphism $e\colon X\ra X$ has a kernel.

Let $\cs$ be a full subcategory of $\ct$. For an object $X$ of $\ct$, a morphism $f:S\ra X$ is called a \emph{right $\cs$-approximation} of $X$ if $S\in\cs $ and $\Hom_{\ct}(S', f)$ is surjective for any $S'\in\cs$. We say that $\cs$ is \emph{contravariantly finite} if every object in $\ct$ has a right $\cs$-approximation. Dually, one defines left $\cs$-approximations and covariantly finite subcategories. We say that $\cs$ is \emph{functorially finite} if it is both contravariantly finite and covariantly finite.

We call $\cs$ a \emph{thick subcategory} of $\ct$ if it is a triangulated subcategory of $\ct$ which is closed under taking direct summands. Denote by $\thick_{\ct}(\cs)$ (or simply $\thick(\cs)$) the smallest thick subcategory of $\ct$ which contains $\cs$. Let $\cs'$ be another full subcategory of $\ct$. Define $$\cs*\cs'=\{X\in\ct\,|\,\text{there is a triangle $S\ra X\ra S'\ra S[1]$ with $S\in\cs$ and $S'\in\cs'$}\}.$$

A full subcategory $ \mathcal{P} $ of $ \mathcal{T} $ is \emph{presilting} if $ \Hom_{\mathcal{T}}(\cp,\cp[i])=0 $ for any $ i>0 $. It is \emph{silting} if in addition $ \mathcal{T}=\thick\mathcal{P} $. An object $ P $ of $ \mathcal{T} $ is \emph{presilting} if $ \add P $ is a presilting subcategory and \emph{silting} if $ \add P $ is a silting subcategory.

For a silting subcategory $\cm$ in $\ct$ satisfying $\cm =\add(\cm)$, we have a co-t-structure $(\ct_{\geqslant0},\ct_{\leqslant0})$ on $\ct$ by \cite[Proposition 2.8]{IY2018}, where 
$$ \mathcal{T}_{\geqslant l}=\mathcal{T}_{>l-1}\coloneqq\bigcup_{i\geqslant0}\mathcal{M}[-1-i]*\cdots*\mathcal{M}[-l-1]*\mathcal{M}[-l],$$
$$ \mathcal{T}_{\leqslant l}=\mathcal{T}_{<l+1}\coloneqq\bigcup_{i\geqslant0}\mathcal{M}[-l]*\mathcal{M}[-l+1]\cdots*\mathcal{M}[-l+i].$$

\subsection{Silting reductions}\label{Section:Reductions}

	
	Let $ \mathcal{P} $ be a presilting subcategory of $ \mathcal{T} $. Let $ \mathcal{S} $ be the thick subcategory $\thick\mathcal{P} $ of $ \ct $ and $ \mathcal{U} $ the quotient category $ \mathcal{T}/\mathcal{S} $. We call $ \mathcal{U} $ the \emph{silting reduction} of $ \mathcal{T} $ with respect to $ \mathcal{P} $ (see \cite{IY2018}). For an integer $ l $, there is a bounded co-t-structure $ (\mathcal{S}_{\geqslant l},\mathcal{S}_{\leqslant l}) $ on $ \mathcal{S} $, where
	$$ \mathcal{S}_{\geqslant l}=\mathcal{S}_{>l-1}\coloneqq\bigcup_{i\geqslant0}\mathcal{P}[-1-i]*\cdots*\mathcal{P}[-l-1]*\mathcal{P}[-l],$$
	$$ \mathcal{S}_{\leqslant l}=\mathcal{S}_{<l+1}\coloneqq\bigcup_{i\geqslant0}\mathcal{P}[-l]*\mathcal{P}[-l+1]\cdots*\mathcal{P}[-l+i].$$
	
	Let $ \cz $ be the following subcategory of $ \ct $
	$$ \mathcal{Z}={}^{\perp_{\mathcal{T}}}(\mathcal{S}_{<0})\cap(\mathcal{S}_{>0})^{\perp_{\mathcal{T}}}={}^{\perp_{\mathcal{T}}}(\mathcal{P}[>0])\cap(\mathcal{P}[<0])^{\perp_{\mathcal{T}}} .$$


	\begin{Lem}\cite[Lemma 3.4]{IY2018}\label{silting reduction}
		The composition $ \mathcal{Z}\subset\mathcal{T}\xrightarrow{\rho}\mathcal{U} $ of natural functors induces a fully faithful embedding
		$$ \overline{\rho}\colon\mathcal{Z}/[\mathcal{P}]\hookrightarrow\mathcal{U}. $$
	\end{Lem}
	
	\begin{Rem}
		In our case, the category $\mathcal{P}$ may not satisfy condition (P1) in \cite[Section 3.1]{IY2018}. We only have a fully faithful embedding
		$
		\overline{\rho}\colon \mathcal{Z}/[\mathcal{P}] \hookrightarrow \mathcal{U}.$
		
	\end{Rem}

	\begin{Thm}\cite[Theorem 4.2]{IYY2008}\label{Tiangule structure for Z/P}
		The category $ \mathcal{Z}/[\mathcal{P}] $ has the structure of a triangulated category with respect to the following shift functor and triangles:
		\begin{itemize}
			\item[(a)] For $ X\in\mathcal{Z} $, we take a triangle $$ X\xrightarrow{l_{X}}P_{X}\longrightarrow X\langle1\rangle\longrightarrow X[1] $$
			with a (fixed) left $ \mathcal{P} $-approximation $ l_{X} $. Then $ \langle1\rangle $ gives a well-defined auto-equivalence of $ \mathcal{Z}/[\mathcal{P}] $, which is the shift functor of $ \mathcal{Z}/[\mathcal{P}] $.
			\item[(b)] For a triangle $ X\longrightarrow Y\longrightarrow Z\longrightarrow X[1] $ with $ X, Y, Z\in\mathcal{Z} $, take the following commutative diagram of triangles
			\begin{align*}
				\xymatrix{
					X\ar[r]^{f}\ar@{=}[d]&Y\ar[r]^{g}\ar[d]&Z\ar[r]^{g}\ar[d]^{a}& X[1]\ar@{=}[d]\\
					X\ar[r]^{l_{X}}&P_{X}\ar[r]&X\langle1\rangle\ar[r]&X[1].
				}
			\end{align*}
			Then we have a complex $ X\xrightarrow{\overline{f}}Y\xrightarrow{\overline{g}}Z\xrightarrow{\overline{a}}X\langle1\rangle .$ We define triangles in $ \mathcal{Z}/[\mathcal{P}] $ as the complexes which are isomorphic to complexes obtained in this way.
		\end{itemize}
	\end{Thm}

	\begin{Thm}\cite[Theorem 3.6]{IY2018}\label{Thm:silting reduction} 
		The functor $ \overline{\rho}\colon \mathcal{Z}/[\mathcal{P}]\longrightarrow\mathcal{U} $ in Lemma~\ref{silting reduction} is a triangle functor where the triangulated structure of $ \mathcal{Z}/[\mathcal{P}] $ is given by Theorem~\ref{Tiangule structure for Z/P}.
	\end{Thm}

	\section{Relative AGK's constructions and Higgs constructions}
	
	Throughout this section, let $k$ be a field and let $D = \operatorname{Hom}_k(-, k)$ denote the $k$-dual. Let $d \geqslant 1$ be an integer. Let $\mathcal{T}$ be a $k$-linear triangulated category and $\mathcal{T}^{\fd}$ a triangulated subcategory of $\mathcal{T}$.

	Let $\mathcal{M}$ be a silting subcategory of $\mathcal{T}$ and $\cp$ a subcategory of $\cm$. Then it is clear that $\cp$ is a pre-silting subcategory of $\ct$. Denote by $$\pi\colon\ct\ra\cu\coloneqq\ct/\thick(\cp)$$ the canonical projection functor. 
	By \cite[Theorem 3.7]{IY2018}, the image $\pi(\cm)$ is also a silting subcategory of $\cu$. By abuse of notation, we will write $\cm$ for $\pi(\cm)$.


	\begin{Def}\rm\label{Def:Calabi--Yau quadruple}
		We say that $(\ct,\ct^{\fd},\cm,\cp)$ is a \emph{$(d+1)$-Calabi--Yau quadruple} if the following conditions are satisfied
		\begin{itemize}
			\item[(CY1)] The category $\ct$ is Krull–Schmidt, $\cp$ is functorially finite in $\cm$ and $\cu=\ct/\thick(\cp)$ is Hom-finite.
			\item[(CY2)] $\thick(\cp)$ is left orthogonal to $\ct^{\mathrm{fd}}$, i.e. $\ct^{\fd}\subseteq\thick(\cp)^{\perp_{\mathcal{T}}}$.
			\item[(CY3)] The pair $(\ct,\ct^{\fd})$ is \emph{relative $(d+1)$-Calabi--Yau} in the sense that there exists a bifunctorial isomorphism
			\[
			D \operatorname{Hom}_{\mathcal{T}}(X, Y) \simeq \operatorname{Hom}_{\mathcal{T}}(Y, X[d+1])
			\]
			for any $X\in\ct^{\fd}$ and $Y\in\ct$.
			\item[(CY4)] $\cu$ admits a $t$-structure $(\cu^{\leqslant0},\cu^{\geqslant0})\coloneqq(\cm[<\!0]^{\perp_{\mathcal{U}}},\cm[>\!0]^{\perp_{\mathcal{U}}})$ with $\cu^{\geqslant0}\subseteq\pi(\ct^{\fd})$.
			\item[(CY5)] $\ct$ admits a $t$-structure $$(\ct^{\leqslant0},\ct^{\geqslant0})$$ 
			with $\ct^{\leqslant0}=\pi^{-1}(\cu^{\leqslant0})$, $\ct^{\geqslant0}=\cm[>\!0]^{\perp_{\mathcal{T}}}\cap\thick(\cp)^{\perp_{\mathcal{T}}}\cap\,^{\perp_{\mathcal{T}}}\thick(\cp)$ and $\ct^{\geqslant0}\subseteq\ct^{\fd}$. Moreover, $\cm$ is a dualizing $k$-variety.
		\end{itemize}
	\end{Def}

We illustrate the t-structure on $\ct$ in the following picture
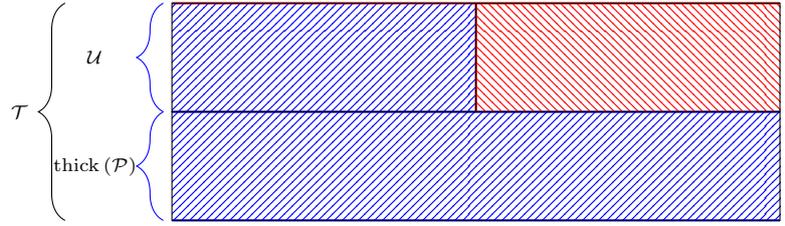
\begin{align*}
	\begin{tikzpicture}[scale=1.6]
		\draw [thick,red] (0,1.8)--(5,1.8);
		\draw [thick,blue] (0,0.9)--(5,0.9);
		\draw [thick,blue] (0,0)--(5,0);
		\draw[thick,red](2.5,1.8)--(2.5,0.9);
		\draw[decorate,decoration={brace,amplitude=10pt},xshift=-25pt,yshift=0pt](0,0)--(0,1.8)node [black,midway,xshift=-0.6cm] {\tiny $\ct$};
		\draw[decorate,blue,decoration={brace,amplitude=10pt},xshift=-2pt,yshift=0pt](0,0)--(0,0.9)node [black,midway,xshift=-0.9cm] {\tiny $\thick(\cp)$};
		\draw[decorate,blue,decoration={brace,amplitude=10pt},xshift=-2pt,yshift=0pt](0,0.9)--(0,1.8)node [black,midway,xshift=-0.9cm] {\tiny $\cu$};
		\draw[pattern=north east lines, pattern color=blue] (0,0) rectangle (5,0.9);
		\draw[pattern=north east lines, pattern color=blue] (0,0.9) rectangle (2.5,1.8);
		\draw[pattern=north west lines, pattern color=red] (2.5,0.9) rectangle (5,1.8);
	\end{tikzpicture}\,,
\end{align*} 
where the blue region represents the subcategory $ \ct^{\leqslant0} $ and the red region represents the subcategory $ \ct^{\geqslant0} $.

\begin{Rem}\label{Rem:silting reduction of CY}\label{Rem:4to3}
	\begin{itemize}
		\item[(1)] If $\cp$ is the zero category, then $(\ct,\ct^{\fd},\cm,0)$ is a $(d+1)$-Calabi--Yau quadruple if and only if $(\ct,\ct^{\fd},\cm)$ is a $(d+1)$-Calabi--Yau triple in the sense of Iyama--Yang \cite[Section 5]{IY2018}.
		\item[(2)] By the relative $(d + 1)$-Calabi--Yau property $\mathrm{(CY3)}$, it implies that $\thick(\cp)$ is also right orthogonal to $\ct^{\mathrm{fd}}$. Hence we have 
		$$\Hom_{\cu}(X,Y)\simeq\Hom_{\ct}(X,Y)$$ and
		$$\Hom_{\cu}(Y,X)\simeq\Hom_{\ct}(Y,X)$$ for any $X\in\ct^{\fd}$ and $Y\in\ct$. This implies that $(\cu,\pi(\ct^{\fd}),\pi(\cm))$ is a $(d+1)$-Calabi--Yau triple. Hence it implies that $\ct^{\fd}\subseteq^{\perp_{\ct}}\!\!(\cp[>\!0])\cap(\cp[<\!0])^{\perp_{\ct}}$.
	\end{itemize}
	
\end{Rem}

\begin{Ex}
	Let $f\colon B\ra A$ be a morphism (not necessarily unital) between two smooth and connective dg $k$-algebras. Let $n$ be a positive integer. Suppose that $f$ has a relative left $n$-Calabi--Yau structure in the sense of Brav--Dyckerhoff \cite{BD2019}. A dg $A$-module is perfectly valued if its
	total cohomology is finite-dimensional and we denote by $\pvd(A)$ the triangulated category
	of perfectly valued dg $A$-modules. 
	
	Let $e=f(\id_{B})$. Assume that $H^{0}(A)$ is Noetherian and $H^{0}(A)/\langle e\rangle$ is finite dimensional. Let $\pvd_{e}(A)$ be the full subcategory of $\pvd(A)$ of the dg $A$-modules whose restriction to $e$ is acyclic. By \cite[Section 4]{Wu2023}, $(\per A,\pvd_{e}(A),\add(A),\add(eA))$ is an $n$-Calabi--Yau quadruple.
\end{Ex}

\begin{Ex}\label{Ex:inter-CY}
Let $A$ be a Noetherian $k$-algebra, $e$ an idempotent of $A$, and $d$ a non-negative
integer. Assume that $A$ is internally $d$-Calabi–Yau with respect to $e$ in the sense of Pressland \cite[Definition 2.1]{Pressland2017} and $A/\langle e\rangle$ is finite dimensional. Then $(\per A,\pvd_{e}(A),\add(A),\add(eA))$ is a $d$-Calabi--Yau quadruple.
	
\end{Ex}	

\begin{Ex}
	Let $A$ be a smooth and connective dg algebra over $k$. Suppose that $A$ is $n$-Calabi--Yau for some positive integer $n$ and $H^{0}(A)$ is Noetherian. Let $e$ be an idempotent of $A$. Assume that $H^{0}(A)/\langle e\rangle$ is finite dimensional. Then $$(\per A,\pvd_{e}(A),\add(A),\add(eA))$$ is also an $n$-Calabi--Yau quadruple. If $e$ is the zero idempotent, then $\pvd_{e}(A)=\pvd(A)$, $\add(eA)=0$ and
	$(\per A,\pvd(A),\add(A))$ is an $n$-Calabi--Yau triple in the sense of Iyama--Yang \cite{IY2018}.
\end{Ex}	
\begin{Ex}\label{Ex:isolated sing}
Let $k$ be a field and $G$ a finite subgroup of $\mathrm{GL}_{n}(k)$.  Assume that the order of $G$ is not divisible by the characteristic of $k$. Denote by $R_{n}$ the formal power series algebra $k[[x_{1},\ldots,x_{n}]]$. Then $R_{n}$ is $n$-Calabi--Yau and $G$ naturally acts on $R$. We remark that $R_{n}$ is the derived $n$-Calabi--Yau completion of $R_{n-1}$, c.f. \cite{BK2011}, \cite[Section 6]{TV2010}.

Denote by $R*G$ the skew group algebra. By \cite[Proposition 3.3.2]{LeMeur2020}, we see that $R*G$ is smooth and $n$-Calabi--Yau. Let $e_{0}=\frac{1}{|G|}\sum_{g\in G}g$. The invariant algebra $R^{G}$ isomorphic to $e_{0}(R*G)e_{0}$ \cite[Lemma 4.6]{Liu2404}. Assume that $R^{G}$ is an isolated singularity. By the Theorem in \cite[pp.200]{Auslander1984}, the stable algebra $R*G/(e_{0})$ is finite dimensional. Hence $$(\per(R*G),\pvd_{e_{0}}(R*G),\add(R*G),\add(e_{0}(R*G)))$$ is an $n$-Calabi--Yau quadruple. 

\end{Ex}

\bigskip

Let $(\ct,\ct^{\fd},\cm,\cp)$ be a $(d+1)$-Calabi--Yau quadruple.
	
	\begin{Lem}\label{Lemma:bounded t-structure on Tfd}
		The pair $(\ct^{\fd}\cap\ct^{\leqslant0},\ct^{\geqslant0})$ is a bounded $t$-structure on $\ct^{\fd}$.
	\end{Lem}
	\begin{proof}
		For $X\in\ct^{\fd}$ , there is a triangle
		\[
		\sigma^{\leqslant 0} X \rightarrow X \rightarrow \sigma^{\geqslant 1} X \rightarrow (\sigma^{\leqslant 0} X)[1].
		\]
		Since both \( X \) and \( \sigma^{\geqslant 1} X \) belong to the triangulated subcategory \( \ct^{\fd} \), it follows that \( \sigma^{\leqslant 0} X \) belongs to \( \ct^{\fd} \) and hence to \( \ct^{\fd} \cap \ct^{\leqslant 0} \). This shows that \( (\ct^{\fd} \cap \ct^{\leqslant 0}, \ct^{\geqslant 0}) \) is a \( t \)-structure on \( \ct^{\fd} \).
		
		Let $X$ be any object of $\ct^{\fd}$. Then we have $\Hom_{\mathcal{T}}(\cm,X[i])=0$ and $\Hom_{\mathcal{U}}(\cm,\pi(X)[-i])\simeq\Hom_{\ct}(\cm,X[-i])\simeq D\Hom_{\ct}(X,\cm[i+d+1])=0$ for $i\gg0$. Namely, $X$ belongs to $\ct^{\fd}\cap\ct^{\leqslant i}\cap\ct^{\geqslant -i}$ for some positive integers $i$.
	\end{proof}

\bigskip
	
	We denote by $\heartsuit$ the heart of the $t$-structure $(\ct^{\leqslant0},\ct^{\geqslant0})$. Denote by $\sigma^{\leqslant i}$ and $\sigma^{\geqslant i+1}$ the truncation functors associated with the $t$-structures $(\ct^{\leqslant i},\ct^{\geqslant i})\coloneqq(\ct^{\leqslant0}[-i],\ct^{\geqslant0}[-i])$. Then for each $X\in\ct^{\leqslant0}$, there exists a triangle in $\ct$
	$$L[-1]\ra Y\ra X\ra L$$
	with $L=\sigma^{\geqslant0}X\in\heartsuit$ and $Y=\sigma^{\leqslant-1}X\in\ct^{\leqslant-1}$.
	
	Notice that the image $\pi(\cm)$ under the quotient functor $\pi\colon\ct\ra\cu=\ct/\thick(\cp)$ is also a silting subcategory of $\cu$.
	\begin{Prop}\label{Prop:triangle with heart and T}
		For each $X\in\pi^{-1}(\cu_{\geqslant0})$, there exists a triangle in $\ct$
		$$L[-d-1]\ra X\ra Y\ra L[-d]$$
	 with $L\in\heartsuit$ and $\pi(Y)\in\cu_{\geqslant1}$. 
	\end{Prop}
	
	\begin{proof}
	Let $X\in\pi^{-1}(\cu_{\geqslant0})$. Then $\pi(X)$ lies in $\cu_{\geqslant0}$. Let $\heartsuit'$ be the heart of $(\cu^{\leqslant
	0},\cu^{\geqslant0})$. It is not hard to see that the quotient functor $\pi\colon\ct\ra\cu$ induces an equivalence $\pi\colon\heartsuit\iso\heartsuit'$ of abelian categories.
	
	By \cite[Proposition 4.12,Lemma 4.13]{IY2018}, there exists an object $L'\in\heartsuit'$ such that there exists a morphism $g'\colon L'[-d-1]\ra\pi(X)$ in $\cu$ which induces a functorial isomorphism
	$$g_{'*}\colon\Hom_{\cu}(\pi(X),U)\iso\Hom_{\cu}(L'[-d-1],U)$$ for $U\in\cu^{\leqslant0}$. Let $L$ be an object of $\heartsuit$ such that $\pi(L)=L'$. 
	
	Since $\heartsuit\subseteq^{\perp_{\ct}}\!\!(\cp[>\!0])\cap(\cp[<\!0])^{\perp_{\ct}}$, we have $\Hom_{\ct}(L[-d-1],X)\cong\Hom_{\cu}(L'[-d-1],\pi(X))$. We lift $g'$ to be a morphism $g\colon L[-d-1]\ra X$ in $\ct$.
	
	We extend the morphism $g$ to a triangle in $\ct$
	$$Y[-1]\ra L[-d-1]\xrightarrow{g} X\ra Y.$$
	
	Then it is enough to show that $\pi(Y)$ lies in $\cu_{\geqslant1}$. By the proof of \cite[Proposition 4.12]{IY2018}, we see that $$\Hom_{\cu}(\pi(Y),\pi(\cm)[\geqslant0])=0.$$ This shows that $\pi(Y)$ lies in $\cu_{\geqslant1}$. 
	
	\end{proof}

	\subsection{The silting reduction of a Calabi--Yau quadruple}
	
	Let $(\ct,\ct^{\fd},\cm,\cp)$ be a $(d+1)$-Calabi--Yau quadruple. Let $\co$ be a functorially finite subcategory of $\cm$.
	
	
	Then $\co$ is a presilting subcategory of $\ct$ satisfying the conditions $(P1)$ and $(P2)$ in Section\ref{Section:Reductions}. Let $$\quad\cv\coloneqq\ct/\thick(\co).$$
	
	Denote by $\beta\colon\ct\ra\cv$ the canonical projection functor. By the relative $(d+1)$-Calabi--Yau property $\mathrm{(CY3)}$, we have $$\ct^{\fd}\cap\thick(\co)^{\perp_{\ct}}=\ct^{\fd}\cap\,^{\perp_{\ct}}\thick(\co),$$
	which will be denoted by $\cv^{\fd}$. This category can be viewed as a full subcategory of $\cv$. 
	
	By abuse of notation, we will write $\cp$ for $\beta(\cp)$. By \cite[Theorem 3.7]{IY2018}, $\frac{\cm}{[\co]}\subseteq\frac{\cz''}{[\co]}\simeq\cv$ is a also silting subcategory of $\cv$. Hence $\cp$ is a presilting subcategory of $\cv$.

	Denote by 
	$$\delta\colon\cv\ra\cv/\thick(\cp)$$ the canonical projection functor. We have the following commutative diagram of triangulated categories
	\[
	\begin{tikzcd}
		&\cu=\ct/\thick(\cp)\arrow[dr,"\alpha"]&\\
		\ct\arrow[ur,"\pi"]\arrow[dr,"\beta",swap]&&\cu/\thick(\co)\simeq\ct/\thick(\co\cup\cp)\simeq\cv/\thick(\cp)\\
		&\cv=\ct/\thick(\co)\arrow[ur,"\delta",swap].&
	\end{tikzcd}
	\]

	Let $$\cz\coloneqq^{\perp_{\ct}}\!\!(\cp[>\!0])\cap(\cp[<\!0])^{\perp_{\ct}}.$$
	Then $\cz\hookrightarrow\ct\ra\ct/\thick(\cp)=\cu$ induces the following triangle equivalence (\cite[Theorems 3.1 and 3.6]{IY2018})
	$$\frac{\cz}{[\cp]}\iso\ct/\thick(\cp)=\cu.$$		
	Notice that we have a triangle equivalence $\frac{\cv}{\thick(\cp)}=\frac{\ct/\thick(\co)}{\thick(\cp)}\iso\frac{\ct}{\thick(\co\cup\cp)}$. Similarly, we have the following triangle equivalences
	$$\frac{\cz'}{[(\co\cup\cp)]}\iso\ct/\thick(\co,\cp)\simeq\cv/\thick(\cp)\simeq\cu/\thick(\co)$$
	and
	$$\frac{\cz''}{[\co]}\iso\ct/\thick(\co)=\cv$$ where $\cz'\coloneqq^{\perp_{\ct}}\!\!((\co\cup\cp)[>\!0])\cap((\co\cup\cp)[<\!0])^{\perp_{\ct}}$ and $\cz''\coloneqq^{\perp_{\ct}}\!\!(\co[>\!0])\cap(\co[<\!0])^{\perp_{\ct}}.$

	\begin{Lem}\cite[Lemma 5.5]{IY2018}\label{Lemm:V=T}
		We have an equality $\cv^{\fd}=\ct^{\fd}\cap\cz''$ of subcategories of $\ct$. In particular, if $\co=\cp$, then we have $\cu^{\fd}=(\ct/\thick(\cp))^{\fd}=\ct^{\fd}$.
	\end{Lem}
	
	\begin{proof}
		Let $X\in\ct^{\fd}$. Then $X\in\cz''$ is equivalent to (see Section \ref{Section:Reductions})
		$$\Hom_{\ct}(X,\thick(\co)_{<0})=0)\quad\text{and}\quad\Hom_{\ct}(\thick(\co)_{>0},X)=0.$$
		By the relative $(d + 1)$-Calabi--Yau property, this is equivalent to $\Hom_{\ct}(\thick(\co)_{<d+1},X)=0)$ and $\Hom_{\ct}(\thick(\co)_{>0},X)=0.$

		Since $\thick(\co) = \thick(\co)_{>0} * \thick(\co)_{\leqslant 0}$ 
		and $\thick(\co)_{\leqslant 0} \subseteq \thick(\co)_{< d+1}$, 
		it follows that the above condition is equivalent to 
		$\Hom_{\ct}(\thick(\co), X) = 0$, i.e. 
		$X \in \thick(\co)^{\perp_{\ct}}$. If $\co=\cp$, by Remark \ref{Rem:4to3}, we have $\cu^{\fd}=\ct^{\fd}\cap\cz=\ct^{\fd}$.

	\end{proof}
	
	\begin{Thm}\label{Thm:quadruple reduction}
		The quadruple $(\cv,\cv^{\fd},\cm,\cp)$ is a $(d+1)$-Calabi--Yau quadruple. Namely,
		\begin{itemize}
			\item[(1)] The category $\cv$ is Krull–Schmidt, $\cp$ is functorially finite in $\cm$ and $\cv/\thick(\cp)$ is Hom-finite.
			\item[(2)] $\thick(\cp)$ is left orthogonal to $\cv^{\mathrm{fd}}$.
			\item[(3)] The pair $(\cv,\cv^{\fd})$ is relative $(d+1)$-Calabi--Yau in the sense that there exists a bifunctorial isomorphism
			\[
			D \operatorname{Hom}_{\cv}(X, Y) \simeq \operatorname{Hom}_{\cv}(Y, X[d+1])
			\]
			for any $X\in\cv^{\fd}$ and $Y\in\cv$.
			\item[(4)] $\frac{\cv}{\thick(\cp)}$ admits a $t$-structure $(\frac{\cv}{\thick(\cp)}^{\leqslant0},\frac{\cv}{\thick(\cp)}^{\geqslant0})\coloneqq(\cm[<\!\!0]^{\perp},\cm[>\!\!0]^{\perp})$ with $\frac{\cv}{\thick(\cp)}^{\geqslant0}\subseteq\delta(\cv^{\fd})$.
			\item[(5)] $\cv$ admits a $t$-structure $$(\cv^{\leqslant0},\cv^{\geqslant0})$$ with $\cv^{\leqslant0}=\delta^{-1}(\frac{\cv}{\thick(\cp)}^{\leqslant0})$, $\cv^{\geqslant0}=\cm[>\!\!0]^{\perp_{\cv}}\cap\thick(\cp)^{\perp_{\cv}}\cap\,^{\perp_{\cv}}\thick(\cp)$ and $\cv^{\geqslant0}\subseteq\cv^{\fd}$. Moreover, $\beta(\cm)$ is a dualizing $k$-variety.
		\end{itemize}
	
	\end{Thm}
	
	\begin{proof}
		
\begin{itemize}
		\item[(1)] The category $\cz''$ is a full subcategory of $\ct$ which is closed under direct summands. Thus it is Krull–Schmidt, so is the additive quotient $\cz''/[\co]$. Notice that for each $X,Y\in\cz'\subseteq\cz$ the space $\Hom_{\frac{\cz'}{[\co\cup\cp]}}(X,Y)$ is a quotient of $\Hom_{\frac{\cz}{[\cp]}}(X,Y)$ and $\cu\simeq\frac{\cz}{[\cp]}$ is Hom-finite. Hence $\cv/\thick(\cp)$ is Hom-finite. 
		\item[(2)] Since $\cv^{\fd}=\ct^{\fd}\cap\thick(\co)^{\perp_{\ct}}=\ct^{\fd}\cap\,^{\perp_{\ct}}\thick(\co)$, we have $$\Hom_{\ct/\thick(\co)}(P,V)=\Hom_{
		\ct}(P,V)=0$$
		for any $P\in\thick(\cp)$ and $V\in\cv^{\fd}\subseteq\ct^{\fd}$. This shows that $\thick(\cp)$ is left orthogonal to $\cv^{\mathrm{fd}}$.

		\item[(3)] Since  $\cv^{\fd}=\ct^{\fd}\cap\thick(\co)^{\perp_{\ct}}=\ct^{\fd}\cap\,^{\perp_{\ct}}\thick(\co)$, we have
		\begin{eqnarray*}
			\begin{split}
				D\Hom_{\cv}(X,Y)=D\Hom_{\ct}(X,Y)\simeq\Hom_{\ct}(Y,X[d+1])\\\simeq\Hom_{\cv}(Y,X[d+1])
			\end{split}
		\end{eqnarray*}
		for any $X\in\cv^{\fd}$ and $Y\in\cv$. 
		
		\item[(4)] By Remark \ref{Rem:4to3} (2), the triple $(\cu,\pi(\ct^{\fd}),\pi(\cm))$ is a $(d+1)$-Calabi--Yau triple. The additive category $\co\cup\cp$ is a presilting subcategory of $\cu$. Let $$(\cu/\thick(\co))^{\fd}\coloneqq\pi(\ct^{\fd})\cap\thick(\co\cup\cp)^{\perp}=\pi(\ct^{\fd})\cap\,^{\perp}\thick(\co\cup\cp).$$
		By \cite[Theorem 5.4]{IY2018}, the triple $(\cu/\thick(\co),(\cu/\thick(\co))^{\fd},\cm)$ is also a $(d + 1)$-Calabi--Yau triple. Hence $\cu/\thick(\co)$ admits a $t$-structure $(\cm[<\!\!0]^{\perp},\cm[>\!\!0]^{\perp})$ with $\cm[>\!\!0]^{\perp}\subseteq(\cu/\thick(\co))^{\fd}=\pi(\ct^{\fd})\cap\thick(\co\cup\cp)^{\perp}$. Notice that the inclusion $\pi(\ct^{\fd})\cap\thick(\co\cup\cp)^{\perp}\subseteq\delta(\cv^{\fd})$ is clear. Hence $(4)$ holds.
	
	\item[(5)] 
	By \cite[Proposition 5.7]{IY2018}, the additive quotient $\rho(\cm)=\cm/[\co]$ is a dualizing $k$-variety. Let $X$ be an object of $\cz''$. Since $\ct$ admits a $t$-structure $(\ct^{\leqslant0},\ct^{\geqslant0})$, there exists a triangle in $\ct$
	\begin{equation}\label{triangle in T}
		\sigma^{\leqslant 0} X \xrightarrow{} X \xrightarrow{} \sigma^{\geqslant 1} X \ra \sigma^{\leqslant 0} X[1]
	\end{equation}
	with $\sigma^{\leqslant 0} X\in\ct^{\leqslant0}=\pi^{-1}(\cu^{\leqslant0})$ and $\sigma^{\geqslant 1} X[1]\in\ct^{\geqslant0}=\cm[>\!0]^{\perp_{\mathcal{T}}}\cap\thick(\cp)^{\perp_{\mathcal{T}}}\cap\,^{\perp_{\mathcal{T}}}\thick(\cp)\subseteq\ct^{\fd}$.
	
	Since $\ct^{\fd}$ is both left and right orthogonal to $\thick(\cp)$, the triangle \eqref{triangle in T} can also be regarded as the canonical triangle of $X \in \cu$ with respect to the $t$-structure $(\cu^{\leqslant 0}, \cu^{\geqslant 0})$. 
	
	By the triangle equivalence $\frac{\cz}{[\cp]}\iso\cu$, there exits an object $X^{\cz}\in\cz$ such that $X\cong X^{\cz}$ in $\cu$. 
	
	Let $\cz^{\cu}\coloneqq^{\perp_{\cu}}\!\!(\co[>\!0])\cap(\co[<\!0])^{\perp_{\cu}}.$ Using the $t$-structure $(\cu^{\leqslant0},\cu^{\geqslant0})$ on $\cu$ and by \cite[Theorem 5.4 (c)]{IY2018}, we have the following triangle 
	\begin{equation}\label{triangle in U}
	\sigma^{\leqslant0}X^{\cz}\ra X^{\cz}\ra\sigma^{\geqslant1}X^{\cz}\ra\sigma^{\leqslant 0} X^{\cz}[1]
    \end{equation}
	with $\sigma^{\leqslant 0} X^{\cz}\in\cz^{\cu}\subseteq\cu^{\leqslant0}$ and $\sigma^{\geq 1}X^{\cz}\in\pi(\ct^{\fd})\cap\cz^{\cu}$. Since $X$ and $X^{\cz}$ are isomorphic in $\cu$, we have $\sigma^{\leqslant0}X\cong\sigma^{\leqslant 0} X^{\cz}$ and $\sigma^{\geqslant1}X\cong\sigma^{\geqslant 1}X^{\cz}$ in $\cu$.

	By \cite[Theorem 5.4 (c)]{IY2018}, the triangle \ref{triangle in U} is also the canonical triangle of $X \in \cu/\thick(\co)$ with respect to the $t$-structure $(\cm[<\!\!0]^{\perp},\cm[>\!\!0]^{\perp})$ on $\cu/\thick(\co)$. This shows that $\alpha(\sigma^{\leqslant 0} X^{\cz})$ belongs to $\frac{\cu}{\thick(\co)}^{\leqslant0}\simeq\frac{\cv}{\thick(\cp)}^{\leqslant0}$ and $\delta(\beta(\sigma^{\leqslant0}X))=\alpha(\pi(\sigma^{\leqslant0}X))\cong\alpha(\sigma^{\leqslant0}X^{Z})$ 
	lies in $\frac{\cu}{\thick(\co)}^{\leqslant0}\simeq\frac{\cv}{\thick(\cp)}^{\leqslant0}$. Hence $\sigma^{\leqslant0}X$ belongs to $\delta^{-1}(\frac{\cv}{\thick(\cp)}^{\leqslant0})$ when we view $\sigma^{\leqslant0}X$ as an object in $\cv=\ct/\thick(\co)$. 
	
	We next show that $\sigma^{\geqslant1}X[1]$ lies in $\cv^{\geqslant0}=\cm[>\!\!0]^{\perp_{\cv}}\cap\thick(\cp)^{\perp_{\cv}}\cap\,^{\perp_{\cv}}\thick(\cp)$. By the argument above, we see that $\sigma^{\geqslant1}X\cong\sigma^{\geqslant 1}X^{\cz}\in\pi(\ct^{\fd})\cap\cz^{\cu}$ in $\cu$. Hence $\sigma^{\geqslant1}X$ belongs to $\cz''\cap\ct^{\fd}=\cv^{\fd}$ (Lemma \ref{Lemm:V=T}) when we view it as an object of $\ct$. By $(2)$, we have $\sigma^{\geqslant1}X[1]\in\thick(\cp)^{\perp_{\cv}}\cap\,^{\perp_{\cv}}\thick(\cp)$.
	
	Since $\cv^{\fd}=\ct^{\fd}\cap\thick(\co)^{\perp_{\ct}}=\ct^{\fd}\cap\,^{\perp_{\ct}}\thick(\co)$, then $$\Hom_{\cv}(\cm[>0],\sigma^{\geqslant1}X[1])=\Hom_{\ct}(\cm[>0],\sigma^{\geqslant1}X[1])=0.$$ 
	Hence we have shown that $\sigma^{\geqslant1}X[1]$ lies in $\cv^{\geqslant0}=\cm[>\!\!0]^{\perp_{\cv}}\cap\thick(\cp)^{\perp_{\cv}}\cap\,^{\perp_{\cv}}\thick(\cp)$. Therefore $\cv=\delta^{-1}(\frac{\cv}{\thick(\cp)}^{\leqslant0})*(\cm[>\!\!0]^{\perp_{\cv}}\cap\thick(\cp)^{\perp_{\cv}}\cap\,^{\perp_{\cv}}\thick(\cp))$.
	
	Notice that 
	$$\Hom_{\cv}(A,B)=\Hom_{\frac{\cv}{\thick(\cp)}}(A,B)=0$$ for each $A\in\delta^{-1}(\frac{\cv}{\thick(\cp)}^{\leqslant0})$ and $B\in\cm[>\!\!0]^{\perp_{\cv}}\cap\thick(\cp)^{\perp_{\cv}}\cap\,^{\perp_{\cv}}\thick(\cp)$.
	
	Summarizing, it is shown that $(\delta^{-1}(\frac{\cv}{\thick(\cp)}^{\leqslant0}),\cm[>\!\!0]^{\perp_{\cv}}\cap\thick(\cp)^{\perp_{\cv}}\cap\,^{\perp_{\cv}}\thick(\cp))$ is a $t$-structure on $\cv$.

	Finally, let $Y\in\cm[>\!\!0]^{\perp_{\cv}}\cap\thick(\cp)^{\perp_{\cv}}\cap\,^{\perp_{\cv}}\thick(\cp)$. The triangle \ref{triangle in T} shows that $Y$ is isomorphic to $\sigma^{\geqslant1}Y$ in $\cv$ and hence belongs to $\cz''\cap\ct^{\fd}=\cv^{\fd}$, as argued above. This shows that $\cm[>\!\!0]^{\perp_{\cv}}\cap\thick(\cp)^{\perp_{\cv}}\cap\,^{\perp_{\cv}}\thick(\cp)\subseteq\cv^{\fd}$.
	
	\end{itemize}
		
	\end{proof}
	
\begin{Cor}\label{Cor:induced Calabi--Yau triple}
	The triple $(\cu,\cu^{\fd},\cm)$ is a $(d+1)$-Calabi--Yau triple.
\end{Cor}	
\begin{proof}
	We apply Theorem~\ref{Thm:quadruple reduction} by setting $\co=\cp$.
\end{proof}

\begin{Def}\rm
	The quadruple $(\cv,\cv^{\fd},\cm,\cp)$ is called the
	\emph{silting reduction of $(\ct,\ct^{\fd},\cm,\cp)$ with respect to $\co$}.
\end{Def}

\begin{Rem}
	When $\cp$ is the zero category, 
	the construction above coincides with the silting reduction 
	of a Calabi--Yau triple as defined in the work of Iyama--Yang \cite[Section 5.2]{IY2018}.  
\end{Rem}

By \cite[Theorem 4.9]{IY2018} and Corollary \ref{Cor:induced Calabi--Yau triple}, the quotient category $\cu$ admits an other $t$-stricture $(^{\perp_{\cu}}\cm[<0],^{\perp_{\cu}}\!\cm[>0])$ with $^{\perp_{\cu}}\cm[<0]\subseteq\cu^{\fd}=\pi(\ct^{\fd})$. 

Let $\ct^{'\leqslant0}=^{\perp_{\ct}}\!\!\cm[<0]\cap\thick(\cp)^{\perp_{\mathcal{T}}}\cap^{\perp_{\mathcal{T}}}\thick(\cp)$ and $\ct^{'\geqslant0}=\pi^{-1}(^{\perp_{\cu}}\!\cm[>0])$. It is clear that $\ct^{'\leqslant0}$ can be viewed as a subcategory of $^{\perp_{\cu}}\cm[<0]$. Hence $\ct^{'\leqslant0}$ is also a subcategory of $\ct^{\fd}$.
	
\begin{Thm}\label{Thm:left t-structure}
	The pair $(\ct^{'\leqslant0},\ct^{'\geqslant0})$ is a $t$-structure on $\ct$ with $\ct^{'\leqslant0}\subseteq\ct^{\fd}$.
\end{Thm}	
	
\begin{proof}
Let $X\in\ct^{'\leqslant-1}$ and $Y\in\ct^{'\geqslant0}$. Then $$\Hom_{\ct}(X,Y)\cong\Hom_{\cu}(\pi(X),\pi(Y)).$$ Since $\pi(X)\in^{\perp_{\cu}}\!\cm[\leqslant0]$ and $\pi(Y)\in^{\perp_{\cu}}\!\cm[>0]$, it shows that $\Hom_{\ct}(X,Y)$ vanishes. So $\Hom_{\ct}(\ct^{'\leqslant-1},\ct^{'\geqslant0})=0$ holds.

Then it is enough to show $\ct=\ct^{'\leqslant-1}*\ct^{'\geqslant0}$. Here $\ct^{'\leqslant-1}=^{\perp_{\ct}}\!\!\cm[\leqslant0]\cap\thick(\cp)^{\perp_{\mathcal{T}}}\cap^{\perp_{\mathcal{T}}}\thick(\cp)$ and $\ct^{'\geqslant0}=\pi^{-1}(^{\perp_{\cu}}\!\cm[>0])=\pi^{-1}(\cu_{\geqslant0})$.

Since $\ct=\bigcup_{l\geqslant0}\ct_{\geqslant-l}$, it is sufficient to show $\ct_{\geqslant -l}\subseteq\ct^{'\leqslant-1}*\ct^{'\geqslant0}$. Using Proposition \ref{Prop:triangle with heart and T} repeatedly, we have
$$\ct_{\geqslant-l}\subseteq\heartsuit[-d-1+l]*\pi^{-1}(\cu_{\geqslant1-l})\subseteq\heartsuit[-d-1+l]*\heartsuit[-d-2+l]*\pi^{-1}(\cu_{\geqslant2-l})\subseteq\cdots.$$
And hence 

$$\ct_{\geqslant-l}\subseteq\heartsuit[-d-1+l]*\heartsuit[-d-2+l]*\cdots*\heartsuit[-d]*\pi^{-1}(\cu_{\geqslant0}).$$
By the relative Serre duality, $\heartsuit[-d-1+l]*\heartsuit[-d-2+l]*\cdots*\heartsuit[-d]\subseteq\ct^{'\leqslant-1}$ holds. Therefore $(\ct^{'\leqslant0},\ct^{'\geqslant0})$ is a $t$-structure on $\ct$.
	
\end{proof}	

\begin{Rem}
When $\cp$ is the zero category, the $t$-structure above is called a left adjacent $t$-structure on $\ct$ \cite[Theorem 4.10]{IY2018}.	
		
\end{Rem}

	\begin{Ex}
		Consider the following ice quiver $(Q,F)$	

		\[
		\begin{tikzcd}
			&2\arrow[dr,"b"]\arrow[dd,dashed]&\\
			1\arrow[ur,"a"]\arrow[d]&&3\arrow[ll,"c",pos=0.55]\arrow[d]\\
			\color{blue}\boxed{4}&\color{blue}\boxed{5}&\color{blue}\boxed{6}\,,
		\end{tikzcd}
		\]
		with potential $W=cba$. The frozen subquiver is given by $F=\{4,5,6\}$.
		Then the associated relative Ginzburg dg algebra $\bm{\Gamma}$ is the following graded quiver
		
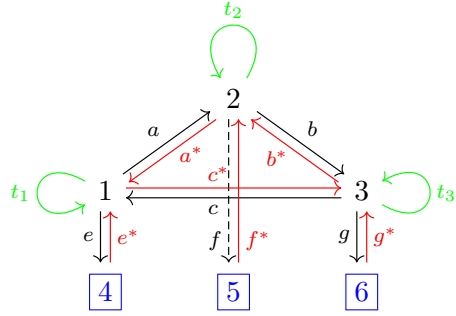
\[
		\begin{tikzcd}
			&2\arrow[dr,"b",shift left=0.7]\arrow[dd,dashed,"f",swap,pos=0.85,shift left=-0.7]\arrow[dl,shift right=-0.4ex,red,"a^{*}"{yshift=4pt}]\arrow[out=60,in=120,loop,green,"t_{2}",swap]&\\
			1\arrow[ur,"a",shift right=-0.4ex]\arrow[d,"e",swap,shift left=-0.7]\arrow[out=150,in=210,loop,green,"t_{1}",swap]\arrow[rr,shift left=0.7,red,"c^{*}",pos=0.43]&&3\arrow[ll,"c",pos=0.595,shift left=0.7]\arrow[d,"g",swap,shift left=-0.7]\arrow[out=330,in=30,loop,green,"t_{3}",swap]\arrow[ul,"b^{*}"{yshift=2.5pt},shift left=0.7,red]\\
			\color{blue}\boxed{4}\arrow[u,"e^{*}",swap,shift left=-0.7,red]&\color{blue}\boxed{5}\arrow[uu,"f^{*}",swap,shift left=-0.7,red,pos=0.15]&\color{blue}\boxed{6}\arrow[u,"g^{*}",swap,shift left=-0.7,red]
		\end{tikzcd}
		\]
		with $|a^{*}|=|b^{*}|=|c^{*}|=|e^{*}|=|f^{*}|=|g^{*}|=-1$ and $|t_{1}|=|t_{2}|=|t_{3}|=-2$.
		The differential $d$ takes the following values
		$$d(a^{*})=cb,\quad d(b^{*})=ac,\quad d(c^{*})=ba,$$
		$$d(e^{*})=0,\quad d(f^{*})=0,\quad d(g^{*})=0,$$
		$$d(t_{1})=-a^{*}a+cc^{*}-e^{*}e,$$
		$$d(t_{2})=aa^{*}-b^{*}b-f^{*}f,$$
		$$d(t_{3})=bb^{*}-c^{*}c-g^{*}g.$$
	Let $\alpha$ be the sum of idempotents associated with frozen vertices. By \cite[Section 4]{Wu2023}, $(\per(\bm{\Gamma}),\pvd_{\alpha}(\bm{\Gamma}),\add(\bm{\Gamma}),\add(\alpha\bm{\Gamma}))$ is a $3$-Calabi--Yau quadruple. 
	
	Let $\beta=e_{2}+e_{5}$. Then the silting reduction of $(\per(\bm{\Gamma}),\pvd_{\alpha}(\bm{\Gamma}),\add(\bm{\Gamma}),\add(\alpha\bm{\Gamma}))$ with respect to $\co=\add(\beta\bm{\Gamma})$ is $(\per(\bm{\Gamma}'),\pvd_{\alpha'}(\bm{\Gamma}'),\add(\bm{\Gamma}'),\add(\alpha'\bm{\Gamma}'))$, where $\bm{\Gamma}'$ is the following dg algebra
	
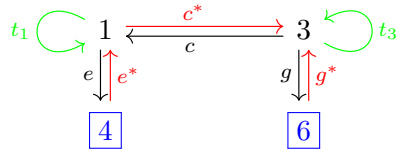
\[
	\begin{tikzcd}
		1\arrow[d,"e",swap,shift left=-0.7]\arrow[out=150,in=210,loop,green,"t_{1}",swap]\arrow[rr,shift left=0.7,red,"c^{*}",pos=0.43]&&3\arrow[ll,"c",pos=0.595,shift left=0.7]\arrow[d,"g",swap,shift left=-0.7]\arrow[out=330,in=30,loop,green,"t_{3}",swap]\\
		\color{blue}\boxed{4}\arrow[u,"e^{*}",swap,shift left=-0.7,red]&&\color{blue}\boxed{6}\arrow[u,"g^{*}",swap,shift left=-0.7,red]
	\end{tikzcd}
	\]
	with differentials $$d(c^{*})=0,\quad d(e^{*})=0,\quad d(g^{*})=0$$
	$$d(t_{1})=cc^{*}-e^{*}e,$$
	$$d(t_{3})=-c^{*}c-g^{*}g$$
and $\alpha'=e_{4}+e_{6}$.
	
	\end{Ex}
		
	\subsection{The relative AGK's relative cluster category and Higgs category of a Calabi--Yau quadruple}\ \label{subsection:AGK relative cluster category}
	
	Let $(\ct,\ct^{\fd},\cm,\cp)$ be a $(d+1)$-Calabi--Yau quadruple.
	\begin{Def}\rm
		The \emph{Amiot--Guo--Keller relative cluster category} of $(\ct,\ct^{\fd},\cm,\cp)$  
		is defined as the following triangle quotient
		$$\cc\coloneqq\ct/\ct^{\fd}.$$
	\end{Def}
	Denote by $\tau\colon\ct\ra\cc$ the canonical projection functor. 
	Let $\cf$ be the following full subcategory of $\ct$
	$$\cf\coloneqq\cm*\cm[1]*\cdots*\cm[d-1]\cap\cz $$ where $$\cz=^{\perp_{\ct}}\!\!(\cp[>\!0])
	\cap(\cp[<\!0])^{\perp_{\ct}}.$$
We call $\cf$ the \emph{relative fundamental domain} associated with $(\ct,\ct^{\fd},\cm,\cp)$.

Recall $\pi$ is the quotient functor $\ct\ra\cu=\ct/\thick(\cp)$ and $\sigma^{\leqslant i}$ and $\sigma^{\geqslant i+1}$ are the truncation functors associated with the $t$-structures $(\ct^{\leqslant i},\ct^{\geqslant i})\coloneqq(\ct^{\leqslant0}[-i],\ct^{\geqslant0}[-i])$.
\begin{Lem}\label{Lem:sigmaX in F}
Let $X$ be an object of $\cz$ such that $\pi(X)$ lies in $^{\perp_{\cu}}\cm[>d])$. Then $\sigma^{\leqslant0}X$ is in $\cf$.
\end{Lem}
\begin{proof}
Consider the triangle
$$\sigma^{\leqslant0}X\ra X\ra\sigma^{\geqslant1}X\ra\sigma^{\leqslant0}X[1].$$ Since $X$ and $\sigma^{\geqslant1}X$ are in $\cz$, then $\sigma^{\leqslant0}X$ also lies in $\cz$.

By \cite[Lemma 5.11]{IY2018}, we see that $\pi(\sigma^{\leqslant0}X)$ lies in $\pi(\cm)*\pi(\cm)[1]*\cdots*\pi(\cm)[d-1]\subseteq\cu$. Then $\pi(\sigma^{\leqslant0}X)$ fits into the following triangles in $\cu$
$$M_{1}\to N_{0}\to \pi(\sigma^{\leqslant0}X)\to \Si M_{1},$$
$$ M_{2}\to N_{1}\to M_{1}\to\Si M_{2},$$
$$ \cdots $$
$$M_{d-2}\to N_{d-3}\to M_{d-3}\to\Si M_{d-2}
$$
with $ N_{0} $, $ N_{1} $, $ \cdots $, $ N_{d-3} $ and $ M_{d-2} $ in $ \pi(\cm) $.

By the $k$-linear equivalence $\pi\colon\cm/[\cp]\iso\pi(\cm)$, we can lift the above tangles to $\ct$, \cite[Proposition 4.6]{KW2307}. This shows that $\sigma^{\leqslant0}X$ is in $\cm*\cm[1]*\cdots*\cm[d-1]$. Hence $\sigma^{\leqslant0}X$ belongs to $\cf$.

\end{proof}

\begin{Rem}
	Let $(Q,F,W)$ be an ice quiver with potential and $\bm{\Gamma}$ the associated relative Ginzburg algebra. Let $e$ be the sum of idempotents associated with frozen vertices and $\pvd_{e}(\bm{\Gamma})$ the thick subcategory of $\per(\bm{\Gamma})$ generated by simple modules associated with all unfrozen vertices. We obtain a 3-Calabi--Yau quadruple $(\per(\bm{\Gamma}),\pvd_{e}(\bm{\Gamma}),\add(\bm{\Gamma}),\add(e\bm{\Gamma}))$. Then the category $\cf$ is exactly the 
	relative fundamental domain of $\per(\bm{\Gamma})$ defined in \cite{Wu2023}.
\end{Rem}

%

The following results generalize the fundamental theorems of 
Amiot~\cite{Am2008}, Guo~\cite{Guo2011}, Wu~\cite{Wu2023}, 
and Iyama--Yang~\cite{IY2018} to the setting of $(d+1)$-Calabi--Yau quadruples.

\begin{Lem}\label{Lem:morohism in C}
	Let $X$ be an object of $\ct$ such that $\pi(X)\in\cu_{\leqslant0}$ and $Y\in\ct$. Let $g$ be an element of $\Hom_{\cc}(X,Y)$. Then it has a representative of the form $X\xleftarrow{s} Z\xrightarrow{f} Y$ such that the cone of $s$ belongs to $\pi^{-1}(\cu_{\leqslant0})\cap\ct^{\fd}$.
\end{Lem}
\begin{proof}
	By definition, the morphism $g$ can be written as $X\xleftarrow{s} Z\xrightarrow{f} Y$ such that there exists a triangle
	$$Z\xrightarrow{s}X\xrightarrow{t}W\ra Z[1]$$
	with $W\in\ct^{\fd}$. For the object $W$, we have a triangle in $\ct$
	$$\sigma^{\leqslant0}W\ra W\ra\sigma^{\geqslant1}W\ra\sigma^{\leqslant0}W[1]$$
	with $\pi(\sigma^{\leqslant0}W)\in\cu^{\leqslant0}$ and $\sigma^{\geqslant1}W\in\cm[>\!0]^{\perp_{\mathcal{T}}}\cap\thick(\cp)^{\perp_{\mathcal{T}}}\cap\,^{\perp_{\mathcal{T}}}\thick(\cp)$.
	Hence $\Hom_{\ct}(X,\sigma^{\geqslant1}W)\cong\Hom_{\cu}(\pi(X),\pi(\sigma^{\geqslant1}W))=0$. Thus $t$ factors through $\sigma^{\leqslant0}W\ra W$. We obtain a commutative diagram of triangles
	\[
	\begin{tikzcd}
		&&\sigma^{\geqslant1}W&\\
		Z\arrow[r,"s"]&X\arrow[r,"t"]&W\arrow[r]\arrow[u]&Z[1]\\
		Z'\arrow[r,"s"]\arrow[u,"h"]&X\arrow[u,equal]\arrow[r]&\sigma^{\leqslant0}W\arrow[r]\arrow[u]&Z'[1]\arrow[u].
	\end{tikzcd}
	\]
By Lemma \ref{Lemma:bounded t-structure on Tfd}, the object $\sigma^{\leqslant0}W$ lies in $\ct^{\fd}$ and $\pi(\sigma^{\leqslant0}W)\in\cu_{\leqslant0}$.	
\end{proof}

\begin{Lem}\label{Lem:faithful functor of F}
The projection functor $\tau\colon\ct\ra\cc$ induces a bijection
$$\Hom_{\ct}(U,V)\ra\Hom_{\cc}(U,V)$$ for any $V\in\ct_{\geqslant1-d}$ and $U\in\ct$ which satisfies $\pi(U)\in\cu_{\leqslant0}$.
		
Consequently, it restricts to a fully faithful functor $\cf\hookrightarrow\cc$.
\end{Lem}
\begin{proof}
Let $U$ be an object of $\ct$ such that $\pi(U)\in\cu_{\leqslant0}$ and $V$ an object of $\ct_{\geqslant1-d}$. We first show that $\Hom_{\ct}(U,V)\ra
\Hom_{\cc}(U,V)$ is injective.

Assume that $f\in\Hom_{\ct}(U,V)$ becomes 
zero in $\cc$. By \cite[Lemma 2.1.26]{Neeman2001}, $f$ factors through some
$W\in\ct^{\fd}$. By Lemma \ref{Lemma:bounded t-structure on Tfd}, there exists
a triangle
$$\sigma^{\leqslant0}W\ra W\ra\sigma^{\geqslant1}W\ra\sigma^{\leqslant0}W[1]$$
such that $\pi(\sigma^{\leqslant0}W)\in\cu^{\leqslant0}$ and
$\sigma^{\geqslant1}W\in\cm[>\!0]^{\perp_{\mathcal{T}}}\cap\thick(\cp)^{\perp_{\mathcal{T}}}\cap\,^{\perp_{\mathcal{T}}}\thick(\cp)$. 

Since $\Hom_{\ct}(U,\sigma^{\geqslant1}W)\cong\Hom_{\cu}(\pi(U),\pi(\sigma^{\geqslant1}W))$ vanishes, the map $f$ also factors through
$\sigma^{\leqslant0}W$ along $\sigma^{\leqslant0}W\ra W$, i.e. we have the following
commutative diagram in $\ct$
\[
\begin{tikzcd}
	U\arrow[rr,"f"]\arrow[d]\arrow[rd]&&V\\
	\sigma^{\leqslant0}W\arrow[r]&W\arrow[r]\arrow[ur]&\sigma^{\geqslant1}W\arrow[r]&\sigma^{\leqslant0}W[1].
\end{tikzcd}
\]

By the relative $(d+1)$-Calabi--Yau property, we have
$$\Hom_{\ct}(\sigma^{\leqslant0}W,V)=D\Hom_{\ct}(V,\sigma^{\leqslant0}W[d+1])
=0$$
as $V\in\ct_{\geqslant1-d}$.
Thus $f$ is zero.

Next we show that $\Hom_{\ct}(U,V)\ra
\Hom_{\cc}(U,V)$ is surjective.
Let $g\colon U\ra V$ be a morphism in $\cc$. By Lemma \ref{Lem:morohism in C}, it has a representative of the form $U\xleftarrow{s}Z\xrightarrow{f}V$ such that the cone $W$ of $s$ belongs to $\pi^{-1}(\cu_{\leqslant0})\cap\ct^{\fd}$. Then we have an exact sequence
$$\Hom_{\ct}(U,V)\xrightarrow{s^{*}}\Hom_{\ct}(Z,V)\ra\Hom_{\ct}(W[-1],V).$$

Since $W[-1]\in\ct^{\fd}\subseteq\thick(\cp)^{\perp_{\ct}}\cap^{\perp_{\ct}}\thick(\cp)$, by the relative $(d+1$)-Calabi--Yau property, we have 
$$\Hom_{\ct}(W[-1],V)\cong D\Hom_{\ct}(V,W[d])\cong D\Hom_{\cu}(\pi(V),\pi(W)[d])=0.$$

The last equality holds since $\pi(V)\in\cu_{\geqslant1-d}$ and $\pi(W)[d]\in\cu_{\leqslant-d}$. Therefore there exits $h\in\Hom_{\ct}(Z,V)$ such that $f=hs$. Hence $U\xleftarrow{s}Z\xrightarrow{f}V$ is equivalent to $h\colon U\ra V$. This implies that $\Hom_{\ct}(U,V)\ra
\Hom_{\cc}(U,V)$ is surjective.

\end{proof}

	
	\begin{Def}\rm
		The \emph{Higgs category} $\ch$ of $(\ct,\ct^{\fd},\cm,\cp)$ is defined as a 
		full subcategory of $\cc$
		$$\ch\coloneqq(\cp[<0])^{\perp_{\cc}}\cap^{\perp_{\cc}}\!(\cp[>0]).$$
	\end{Def}
\begin{Rem}
The category $\mathcal{P}$ may not satisfy condition (P1) in \cite[Section 3.1]{IY2018}. 
We only have a fully faithful embedding 
\[
\mathcal{H}/[\mathcal{P}] \hookrightarrow \mathcal{U}/\mathcal{U}^{\mathrm{fd}}.
\]
However, since $\mathcal{U}$ is Hom-finite by our assumption, we will show in Theorem~\ref{Thm:Higgs is Fro} that this embedding is an equivalence.	
\end{Rem}

\begin{Thm}\label{Thm:F=H}
	The functor $\tau\colon\ct\ra\cc$ induces an equivalence of additive categories.
	$$\tau\colon\cf\iso\ch.$$
\end{Thm}	
\begin{proof}
		By Lemma~\ref{Lem:faithful functor of F}, we have a fully faithful embedding $\tau\colon\cf\hookrightarrow\cc$. Since $\cp$ is left and right orthogonal to $\ct^{\fd}$, the image $\tau(\cf)$ lies in $\ch$. Hence $\tau\colon\ct\ra\cc$ restricts to a fully faithful functor $\cf\hookrightarrow\ch$. It is enough to show that it is dense. 
		
		Let $X$ be an object of $\ch$ and view it as an object of $\ct$. Since $\thick(\cp)$ and $\ct^{\fd}$ are left orthogonal and right orthogonal to each other, we see that $X$ lies in $\cz=^{\perp_{\ct}}\!\!(\cp[>\!0])
		\cap(\cp[<\!0])^{\perp_{\ct}}\subseteq\ct$.

		By Theorem \ref{Thm:left t-structure}, we have $t$-structure $(\ct^{'\leqslant0},\ct^{'\geqslant0})$ on $\ct$ where 
		$\ct^{'\leqslant0}=^{\perp_{\ct}}\!\!\cm[<0]\cap\thick(\cp)^{\perp_{\mathcal{T}}}\cap^{\perp_{\mathcal{T}}}\thick(\cp)$ and $\ct^{'\geqslant0}=\pi^{-1}(^{\perp_{\cu}}\!\cm[>0])$. This gives a triangle in $\ct$
		$$Y\ra X\ra Z\ra Y[1]$$
		with $Y\in^{\perp_{\ct}}\!\!\cm[<d]\cap\thick(\cp)^{\perp_{\mathcal{T}}}\cap^{\perp_{\mathcal{T}}}\thick(\cp)\subseteq\ct^{\fd}$ and $Z\in\pi^{-1}(^{\perp_{\cu}}\!\cm[>d])$. Then we have $X\simeq Z$ in $\cc$ and $Z$ lies in $\cz$.
		
		Since $(\ct^{\leqslant0},\ct^{\geqslant0})$ is also a $t$-structure on $\ct$, there exits a triangle
		$$\sigma^{\leqslant0}Z\ra Z\ra\sigma^{\geqslant1}Z\ra\sigma^{\leqslant0}Z[1]$$
		with $\sigma^{\leqslant0}Z\in\ct^{\leqslant0}=\pi^{-1}(\cu^{\leqslant0})$ and $\sigma^{\geqslant1}Z\in\ct^{\geqslant1}=\cm[\geqslant\!0]^{\perp_{\mathcal{T}}}\cap\thick(\cp)^{\perp_{\mathcal{T}}}\cap\,^{\perp_{\mathcal{T}}}\thick(\cp)\subseteq\ct^{\fd}$. It is easy to see that $\sigma^{\leqslant0}Z$ also lies in $\cz$ and $Z\simeq\sigma^{\leqslant0}Z$ in $\cc$. Since $\pi(Z)\in^{\perp_{\cu}}\!\!\cm[>d]$, by Lemma \ref{Lem:sigmaX in F}, the object $\sigma^{\leqslant0}Z$ is in $\cf$. Thus the assertion follows.
		
%
%
%
%
%
%
%
%
%
%
%
%
%
%
%

\end{proof}

It is clear that $\ch$ is an extension closed subcategory of
$\cc$. Hence it carries a canonical extriangulated structure in the sense of Nakaoka-Palu \cite{NP2019} and $(\ch,\mathbb{E},\mathfrak{s})$ be described as follows$\colon$
\begin{itemize}
	\item[(1)]  For any two objects $X,Y\in\ch$, the $\mathbb{E}$-extension space $\mathbb{E}(X,Y)$ is given by $\Hom_{\cc}(X,Y[1])$.
	\item[(2)] For any $\delta\in\mathbb{E}(X,Y)=\Hom_{\cc}(X,Y[1])$, take a distinguished triangle
	$$X\xrightarrow{f}Y\xrightarrow{g}Z\xrightarrow{h}X[1]$$ and define $\mathfrak{s}(\delta)=[X\xrightarrow{f}Y\xrightarrow{g}Z]$. This $\mathfrak{s}(\delta)$ does not depend on the choice of the distinguished
	triangle above. 
\end{itemize}

%
%

\begin{Thm}\label{Thm:Higgs is Fro}
\begin{itemize}
	\item[(1)] The Higgs category $\ch$ is a Krull--Schmidt Frobenius extriangulated category with projective-injective objects $\cp$. And $\cm$ is a $d$-cluster-tilting subcategory of $\ch$, i.e. $\cm$ is is functorially finite in $\ch$ and for any $X\in\ch$, the following are equivalent
	\begin{itemize}
		\item[(a)] $X\in\ch$;
		\item[(b)] $\Hom_{\cc}(X,\cm[i])=0$ for $1\leqslant i\leqslant d-1$;
		\item[(c)] $\Hom_{\cc}(\cm,X[i])=0$ for $1\leqslant i\leqslant d-1$;
	\end{itemize}
	\item[(2)] We have the following equivalence of triangulated categories  $$\underline{\ch}=\ch/[\cp]\iso\cu/\cu^{\fd},$$ where $\cu=\ct/\thick(\cp)$. And $\cu/\cu^{\fd}$ is $d$-Calabi--Yau.
\end{itemize}
\end{Thm}
\begin{proof}
Let $I$ be an object in $\cp$. For any distinguished triangle in $\ch$
$$X\ra Y\ra Z\dasharrow,$$
the space $\mathbb{E}(Z,I)=\Hom_{\cc}(Z[-1],I)=0$. Hence we have the following exact sequence
$$\Hom_{\ch}(Y,I)\ra\Hom_{\ch}(X,I)\ra0.$$ Thus any object in $\cp$ is injective. Now let $X$ be an object in $\ch$. Since $\cp$ is functorially finite in $\ct$, there exists a triangle in $\ct$
$$X\xrightarrow{l_{X}} I_{X}\ra X'\ra X[1]$$
with $l_{X}$ a left $\cp$-approximation and $X'\in\cz$.

By Theorem \ref{Thm:left t-structure}, we have $t$-structure $(\ct^{'\leqslant0},\ct^{'\geqslant0})$ on $\ct$ where 
$\ct^{'\leqslant0}=^{\perp_{\ct}}\!\!\cm[<0]\cap\thick(\cp)^{\perp_{\mathcal{T}}}\cap^{\perp_{\mathcal{T}}}\thick(\cp)$ and $\ct^{'\geqslant0}=\pi^{-1}(^{\perp_{\cu}}\!\cm[>0])$. This gives a triangle in $\ct$
$$Y'\ra X'\ra Z'\ra Y[1]$$
with $Y'\in^{\perp_{\ct}}\!\!\cm[<d]\cap\thick(\cp)^{\perp_{\mathcal{T}}}\cap^{\perp_{\mathcal{T}}}\thick(\cp)\subseteq\ct^{\fd}$ and $Z'\in\pi^{-1}(^{\perp_{\cu}}\!\cm[>d])$. Then we have $X\simeq Z'$ in $\cc$ and $Z'$ lies in $\cz$.

Since $(\ct^{\leqslant0},\ct^{\geqslant0})$ is also a $t$-structure on $\ct$, there exits a triangle
$$\sigma^{\leqslant0}Z'\ra Z'\ra\sigma^{\geqslant1}Z'\ra\sigma^{\leqslant0}Z'[1]$$
with $\sigma^{\leqslant0}Z'\in\ct^{\leqslant0}=\pi^{-1}(\cu^{\leqslant0})$ and $\sigma^{\geqslant1}Z'\in\ct^{\geqslant1}=\cm[\geqslant\!0]^{\perp_{\mathcal{T}}}\cap\thick(\cp)^{\perp_{\mathcal{T}}}\cap\,^{\perp_{\mathcal{T}}}\thick(\cp)\subseteq\ct^{\fd}$. It is easy to see that $\sigma^{\leqslant0}Z'$ also lies in $\cz$ and $Z'\simeq\sigma^{\leqslant0}Z'$ in $\cc$. Since $\pi(Z')\in^{\perp_{\cu}}\!\!\cm[>d]$, by Lemma \ref{Lem:sigmaX in F}, the object $\sigma^{\leqslant0}Z'$ is in $\cf$. Thus $\sigma^{\leqslant0}Z'$ lies in $\cf$. 

And we have a triangle in $\cc$
$$X\xrightarrow{l_{X}}I_{X}\ra \sigma^{\leqslant0}Z'\ra X[1].$$
This shows that $l_{X}$ is an inflation in $\ch$. Therefore $\ch$ has has enough injectives.

It remains to show that any injective object is in $\cp$.  Let $J$ be an injective object in $\ch$. We
take a triangle in $\ct$
$$J\xrightarrow{l_{J}}I_{J}\ra J'\ra J[1]$$
with a left $\cp$-approximation $l_{J}$ and $J'\in\cz$.  Since $J$ is injective, the morphism $l_{J}$ is split in $\ch$.  Thus $l_{J}$ is also split in $\cz$. Therefore $J$ belongs to $\cp$ and the subcategory of injective objects in $\ch$ is $\cp$. Similarly, we show that $\ch$ has has enough projectives and the full subcategory of projective objects in $\ch$ is $\cp$. Hence $\ch$ is a Frobenius extriangulated category with projective-injective objects $\cp$.

By an argument analogous to \cite[Proposition 4.7, Proposition 4.8 and Corollary 4.9]{KW2307}, we have the equivalence between triangulated
categories $\ch/[\cp]\simeq\cu/\cu^{\fd}$, $\cu/\cu^{\fd}$ is $d$-Calabi--Yau and $\cm$ is $d$-cluster-tilting subcategory of $\ch$.
	
\end{proof}

\begin{Cor}\cite[Theorem 6.2]{Wu2023}\cite[Theorem 4.18]{KW2307}\label{Cor:mainresult}
Let $A$	be a smooth and connective dg algebra over $k$ and $e$ an idempotent of $H^{0}(A)$. Let $n$ be a positive integer. Assume that $(\per A,\pvd_{e}(A),\add A,\add(eA))$ is a $(n+1)$-Calabi--Yau quadruple. Then we have
\begin{itemize}
	\item[(1)] The associated Higgs category $\ch$ is a Krull--Schmidt Frobenius $n$-Calabi--Yau extriangulated category with projective-injective objects $\add(eA)$. And the object $A$ is a canonical  $n$-cluster-tilting object of $\ch$ with endomorphism algebra $\End_{\ch}(A)\cong H^{0}(A)$.
	\item[(2)] The stable category $\underline{\ch}=\ch/[\add(eA)]$ is triangle equivalent to the generalized cluster category $\per(\overline{A})/\pvd(\overline{A})$, where $\overline{A}$ is the Drinfeld dg quotient of $A$ by $eAe$.

	\item[(3)] If moreover $A$ is concentrated in degree $0$, i,e. $A\iso H^{0}(A)$ is a quasi-isomorphism and $H^{0}(A)$ is Noetherian, then the boundary algebra $B= eH^{0}(A)e$ is Iwanaga--Gorenstein of injective dimension at most $(n+1)$ and the Higgs category $\ch$ is equivalent
	to the category $\mathrm{gpr}(B)$ of Gorenstein projective modules over $B$ and the relative cluster category $\per A/\pvd_{e}(A)$ is equivalent to the derived category $\cd^{b}(\mathrm{gpr}(B))$.
	
	\item[(4)] If $A$ is proper, i.e. $\sum_{i\in\mathbb{Z}}\dim_{k}H^{i}(A)<\infty$, then the boundary dg algebra $B= eAe$ is Iwanaga--Gorenstein in the sense of Jin \cite{Jin2020}, i.e. $\!$the thick subcategory $\per(A)$ of the derived category $\cd(A)$
	coincides with the thick subcategory  $\thick(DA)$ generated by $DA$, where $D=\Hom_{k}(?, k)$ is the $k$-dual. And the Higgs category $\ch$ is equivalent
	to the category $\mathrm{CM}(B)$ of Cohen-Macaulay dg $B$-modules \cite[Definition 0.2]{Jin2020} and the relative cluster category $\per A/\pvd_{e}(A)$ is equivalent to $\pvd(B)$.

\end{itemize}
\end{Cor}

\begin{proof}
The statements $(1)$ and $(2)$ follow from Theorem \ref{Thm:Higgs is Fro}. 

(3) Let $B=eH^{0}(A)e$. Since $A$ is concentrated in degree 0 and $A$ is smooth, it is not hard to see that $B$ is Iwanaga--Gorenstein of injective dimension at most $(n+1)$. Notice that the exact sequence of triangulated categories 
$$0\ra\pvd_{e}(A)\ra\per A\ra\per(A)/\pvd_{e}(A)\ra0$$
is equivalent to  
$$0\ra\pvd(\overline{A})\ra\cd^{b}(\mathrm{mod}H^{0}(A))\ra\cd^{b}(\mathrm{mod}(B))\ra0.$$
Here $\cd^{b}(\mathrm{mod}H^{0}(A))$and $\cd^{b}(\mathrm{mod}(B)$ are the bounded derived categories of finitely generated modules. By \cite[Theorem 3.10]{IY2018}, the Higgs category $\ch$ is equivalent to  $$\mathrm{gpr}(B)=\{X\in\mathrm{mod}B\,|\,\Ext^{i}_{B}(X,B)=0\,\,\text{for any $i>0$}\},$$ the category of Gorenstein projective $B$-modules. By \cite[Lemma 2]{Palu2009}, we see that $\cd^{b}(\mathrm{mod}B)$ is also equivalent to $\cd^{b}(\mathrm{gpr}B)$. The proof of $(4)$ will appear in \cite{ChenDing2026}.

%

%
%
%
%
%
%
%
	
\end{proof}

\begin{Rem}
Under the setting of Example \ref{Ex:inter-CY}, the results (1) and (3) in Corollary \ref{Cor:mainresult} were proved by Pressland \cite[Theorem 4.1, Theorem 4.10]{Pressland2017}. Our proof uses a different approach, we first establish the statements in the Higgs category and then transfer them to $\mathrm{gpr}(B)$.


\end{Rem}

\begin{Ex}\label{Ex:isolated singcat}
Let us continue Example \ref{Ex:isolated sing}.	The quadruple $(\per(R*G),\pvd_{e_{0}}(R*G),\add(R*G),\add(e_{0}(R*G)))$ is an $n$-Calabi--Yau quadruple.
Denote by $\ch(R^{G})$ the corresponding Higgs category. Let $\underline{R*G}$ be the Drinfeld dg quotient of $R*G$ by $e_{0}(R*G)e_{0}$. By \cite[Corollary 7.1]{BD2019}, the dg quotient $\underline{R*G}$ is also $n$-Calabi--Yau and $H^{0}(\underline{R*G})$ is isomorphic to the stable algebra $R*G/(e_{0})$ which is finite dimensional.

The quotient $\per(R*G)/\thick(e_{0}(R*G))$ is equivalent to $\per(\underline{R*G})$. By Theorem \ref{Thm:Higgs is Fro}, we have the following equivalence of triangulated categories
$$\underline{\ch(R^{G})}\iso\frac{\per(\underline{R*G})}{\pvd(\underline{R*G})}.$$

Notice that the quotient $\frac{\per(\underline{R*G})}{\pvd(\underline{R*G})}$ is the generalized cluster category of $\underline{R*G}$ which is $(n-1)$-Calabi--Yau. In fact, the Higgs category $\ch(R^{G})$ is a Frobenius Quillen exact category and is equivalent to $\mathrm{gpr}(R^{G})$, the category of Gorenstein projective modules over $R^{G}$ \cite[Theorem 6.2]{Wu2023}. Consequently, we obtain the following equivalence of triangulated categories (\cite[Theorem 1.1]{TV2010}, \cite[Theorem 5.1, Corollary 5.2, Corollary 5.3]{Am2015} \cite[Proposition 6.13]{KY2016}, \cite[Corollary 10.5]{KY2018}, \cite[Theorem 4.5, Theorem 4.10]{Liu2404})
$$\underline{\mathrm{gpr}(R^{G})}\simeq\frac{\per(\underline{R*G})}{\pvd(\underline{R*G})}.$$
	
\end{Ex}

%
%
%
\bigskip

\section{Reductions of Higgs categories}
Let $(\ct,\ct^{\fd},\cm,\cp)$ be a $(d+1)$-Calabi--Yau quadruple. Denote by $\ch$ the corresponding Higgs category. By Theorem \ref{Thm:Higgs is Fro}, $\ch$ is a Frobenius extriangulated category with projective-injective objects $\cp$ and we have the following equivalence of triangulated categories  $$\underline{\ch}=\ch/[\cp]\iso\cu/\cu^{\fd},$$ where $\cu=\ct/\thick(\cp)$. 

\bigskip

Let $\cq$ be a functorially finite subcategory of $\cm$. Let $\cv=\ct/\thick(\cq)$ and $\cv^{\fd}=\ct^{\fd}\cap\thick(\cq)^{\perp_{\ct}}=\ct^{\fd}\cap^{\perp_{\ct}}\!\thick(\cq)$. Denote by $\pi_{\cq}\colon\ct\ra\cv$ and $\tau_{\cq}\colon\cv\ra\cv/\cv^{\fd}$ the quotient functors.

By Theorem 
\ref{Thm:quadruple reduction}, we obtain a $(d+1)$-Calabi--Yau quadruple $(\cv,\cv^{\fd},\pi_{\cq}(\cm),\pi_{\cq}(\cp))$. By abuse of notation, we write $\cm$ and $\cp$ for $\pi_{\cq}(\cm)$ and $\pi_{\cq}(\cp)$.

Let $\cf_{\cq}=\cm*\cm[1]*\cdots*\cm[d-1]\cap^{\perp_{\cv}}\!\!(\cp[>\!0])
\cap(\cp[<\!0])^{\perp_{\cv}}
\subseteq\cv$. Let $\cc_{\cq}=\cv/\cv^{\fd}$. The corresponding Higgs category $\ch_{\cq}$ is given by 
$$\ch_{\cq}=(\cp[<0])^{\perp_{\cc_{\cq}}}\cap^{\perp_{\cc_{\cq}}}\!(\cp[>0]).$$

By Theorem \ref{Thm:F=H} and Theorem \ref{Thm:Higgs is Fro},
we have a $k$-linear equivalence $\tau_{\cq}\colon\cf_{\cq}\iso\ch_{\cq}$ and $\ch_{\cq}$ is a Frobenius extriangulated category with projective-injective objects $\pi_{\cq}(\cp)$ and we have the following equivalence of triangulated categories  $$\underline{\ch_{\cq}}=\ch_{\cq}/[\cp]\iso\cu/\cu^{\fd},$$ where $\cu=\ct/\thick(\cp)$ and $\cu^{\fd}=\ct^{\fd}\cap\thick(\cp)^{\perp_{\ct}}$.


Let $\ch'_{\cq}$ be the following extension closed subcategory of $\ch$
$$\ch'_{\cq}=\{M\in\ch\,|\,\Hom_{\cc}(M,\cq[1]*\cq[2]*\cdots\cq[d-1])=0\}\subseteq\ch.$$

\begin{Def}\rm\label{Def:Calabi--Yau reduction}
	The \emph{Calabi--Yau reduction of $\ch$ with respect to $\cq$} is defined as the following additive quotient
	$$\frac{\ch'_{\cq}}{[\cq]}.$$
\end{Def}
We will later show that $\ch'_{\cq}$ is also a Frobenius extriangulated category with projective-injective objects $\cp\cup\cq$.

Denote by $\cz_{\cq}=^{\perp_{\ct}}\!\!\!(\cq[>0])\cap(\cq[<0])^{\perp_{\ct}}\subseteq\ct$. The quotient functor $\pi_{\cq}$ induces a triangle equivalence of triangulated categories (\cite[Theorem 3.1 and 3.6]{IY2018})
$$\pi_{\cq}\colon\frac{\cz_{\cq}}{[\cq]}\iso\cv=\ct/\thick(\cq).$$

For $X\in\ct$, we have a triangle in $\ct$
\begin{equation}\label{triangleof t-structure}
	\sigma^{\leqslant0}X\ra X\ra\sigma^{\geqslant1}X\ra\sigma^{\leqslant0}X[1]
\end{equation}
such that $\sigma^{\leqslant0}X\in\ct^{\leqslant0}$ and $\sigma^{\geqslant1}X\in\ct^{\geqslant1}\subseteq\ct^{\fd}$.

\begin{Lem}\label{Lem:>1X lies in Z}\cite[Lemma 5.6]{IY2018}
Let $X\in\cz_{\cq}$. Then $\sigma^{\geqslant1}X\in\cv^{\fd}$ and $\sigma^{\leqslant0}X\in\cz_{\cq}$.
\end{Lem}

\begin{proof}
Since $\sigma^{\geqslant1}X$ belongs to $\ct^{\fd}\subseteq\thick(\cp)^{\perp_{\ct}}\cap^{\perp_{\ct}}\!\thick(\cp)$, we have 
$$\Hom_{\ct}(\sigma^{\geqslant1}X,\cq[k])\cong\Hom_{\cu}(\sigma^{\geqslant1}X,\cq[k])$$
and $$\Hom_{\ct}(\cq[l],\sigma^{\geqslant1}X)\cong\Hom_{\cu}(\cq[l],\sigma^{\geqslant1}X)$$ with $k,l\in\mathbb{Z}$.

Under the quotient functor $\pi\colon\ct\ra\cu$, the triangle \ref{triangleof t-structure} becomes the canonical triangle associated with the $t$-structure $(\cu^{\leqslant0},\cu^{\geqslant0})$. And $\pi(\cq)$ is a presilting subcategory of $\pi(\cm)$. By \cite[Lemma 5.6]{IY2018}, we see that $\pi(\sigma^{\geqslant1}X)$ lies in $\thick(\pi(\cq))^{\perp_{\cu}}\cap^{\perp_{\cu}}\!\thick(\pi(\cq))^{\perp_{\cu}}$. Hence $\sigma^{\geqslant1}X$ is in $\ct^{\fd}\cap\thick(\cq)^{\perp_{\ct}}\cap^{\perp_{\ct}}\!\thick(\cq)=\cv^{\fd}$. 

And $\sigma^{\geqslant1}X[-1]$ also lies in $\cv^{\fd}$. Then it is clear that $\sigma^{\leqslant0}X\in\cz_{\cq}$.
\end{proof}

\bigskip

Recall from Subsection \ref{subsection:AGK relative cluster category} that the category $\cf$ is given by $$\cf\coloneqq\cm*\cm[1]*\cdots*\cm[d-1]\cap\cz $$ where $$\cz=^{\perp_{\ct}}\!\!(\cp[>\!0])
\cap(\cp[<\!0])^{\perp_{\ct}}.$$

\begin{Lem}\label{Lem:ZF equivalence to H'}
	The functor $\tau\colon\ct\ra\cc$ induces a $k$-linear equivalence $\tau\colon\cz_{\cq}\cap\cf\iso\ch'_{\cq}$.
\end{Lem}
\begin{proof}
We first show that for any $X\in\cz_{\cq}\cap\cf$ and $1\leqslant i\leqslant d-1$, the space $\Hom_{\cc}(X,\cq[i])$ vanishes. Then by Theorem \ref{Thm:F=H}, the functor 
$\tau$ induces a well defined functor $\cz_{\cq}\cap\cf\ra\ch'_{\cq}$.

Let $X\in\cz_{\cq}\cap\cf$ and $1\leqslant i\leqslant d-1$. We have a triangle in $\ct$
$$\sigma^{\leqslant0}X\xrightarrow{f} X\xrightarrow{g}\sigma^{\geqslant1}X\ra\sigma^{\leqslant0}X[1]$$
such that $\sigma^{\leqslant0}X\in\ct^{\leqslant0}$ and $\sigma^{\geqslant1}X\in\ct^{\geqslant1}\subseteq\ct^{\fd}$.

By Lemma \ref{Lem:>1X lies in Z}, the object $\sigma^{\geqslant1}X$ lies in $\ct^{\fd}\cap\thick(\cq)^{\perp_{\ct}}\cap^{\perp_{\ct}}\!\thick(\cq)$ and $\sigma^{\leqslant0}X\in\cz_{\cq}$. Therefore the induced map $$f^{*}\colon\Hom_{\ct}(X,\cq[i])\ra\Hom_{\ct}(\sigma^{\leqslant0}X,\cq[i])$$ is a bijection.

Since $\sigma^{\leqslant0}X$ and $X$ are isomorphic in $\cc$, we have a bijection
$$\Hom_{\cc}(X,\cq[i])\cong\Hom_{\cc}(\sigma^{\leqslant0}X,\cq[i]).$$

Notice that $\pi(\sigma^{\leqslant0}X)\in\cu_{\leqslant0}$ and $\cq[i]\subseteq\ct_{\geqslant1-d}$. By Lemma \ref{Lem:faithful functor of F}, we have $\Hom_{\ct}(\sigma^{\leqslant0}X,\cq[i])\cong\Hom_{\cc}(\sigma^{\leqslant0}X,\cq[i])$. Hence we have $\Hom_{\ct}(X,\cq[i])\cong\Hom_{\cc}(X,\cq[i])=0$. By Theorem \ref{Thm:F=H}, the functor 
$\tau$ induces a well defined functor $\cz_{\cq}\cap\cf\ra\ch'_{\cq}$.

We next show that $\tau\colon\cz_{\cq}\cap\cf\ra\ch'_{\cq}$ is dense. 

Let $N$ be an object of $\ch'_{\cq}\subseteq\ch$. By Theorem \ref{Thm:F=H}, there exists an object $N'\in\cf$ such that $\tau(N')\cong N$ in $\cc$.

Since $N'$ lies in $\cf\subseteq\cm*\cm[1]*\ldots\cm[d-1]$, we have
$$\Hom_{\ct}(N',\cq[\geqslant d])=0$$
and 
$$\Hom_{\ct}(\cq[<0],N')=0.$$

By Lemma \ref{Lem:faithful functor of F}, we have
$$\Hom_{\ct}(N',\cq[i])\cong\Hom_{\cc}(N',\cq[i])=0$$
for $1\leqslant i\leqslant d-1$. Thus $N'$ lies in $\cz_{\cq}\cap\cf$. Hence $\tau\colon\cz_{\cq}\cap\cf\ra\ch'_{\cq}$ is dense. Then it is clear that $\tau\colon\cz_{\cq}\cap\cf\ra\ch'_{\cq}$ is a $k$-linear equivalence.
	
\end{proof}

We obtain a $k$-linear equivalence  $\tau\colon\cz_{\cq}\cap\cf\iso\ch'_{\cq}$ and further induces an additive functor
$$\tau\colon\frac{\cz_{\cq}\cap\cf}{[\cq]}\ra\frac{\ch'_{\cq}}{[\cq]}.$$

\begin{Lem}\label{Lem:ZQF=H'Q}
	The functor $\tau\colon\frac{\cz_{\cq}\cap\cf}{[\cq]}\ra\frac{\ch'_{\cq}}{[\cq]}$ is a $k$-linear equivalence.
\end{Lem}
\begin{proof}
	It follows from Lemma \ref{Lem:ZF equivalence to H'}.
\end{proof}

\begin{Lem}\label{Lem:reduction of F}
	The quotient functor $\pi_{\cq}\colon\ct\ra\cv=\ct/\thick(\cq)$ induces an equivalence of extriangulated categories
	$$\pi_{\cq}\colon\frac{\cz_{\cq}\cap\cf}{[\cq]}\iso\cf_{\cq}.$$
\end{Lem}

\begin{proof}
	Let $X$ be an object of $\cz_{\cq}\cap\cf$. It is clear that $\pi_{\cq}(X)$ lies in $\cm*\cm[1]*\ldots*\cm[d-1]\subseteq\cv$. By \cite[Lemma Lemma 3.4]{IY2018}, we see that $\pi_{\cq}(X)$ lies in $^{\perp_{\cv}}\!(\cp[>\!0])
	\cap(\cp[<\!0])^{\perp_{\cv}}$. Hence $\pi_{\cq}\colon\ct\ra\cv$ induces a $k$-linear functor $\pi_{\cq}\colon\frac{\cz_{\cq}\cap\cf}{[\cq]}\ra\cf_{\cq}.$
	
	Since the quotient functor $\pi_{\cq}\colon\ct\ra\cv$ induces an triangle equivalence of triangulated categories
	$$\pi_{\cq}\colon\cz_{\cq}/[\cq]\iso\cv.$$
	Hence it induces an equivalence of extriangulated categories
	$$\pi_{\cq}\colon\frac{\cz_{\cq}\cap\cf}{[\cq]}\iso\cf_{\cq}.$$
	
%
%
%
%
%
%
%
%
%
%
\end{proof}

\bigskip

By definition, the category $\ch'_{\cq}$ is an extension closed subcategory of $\ch$. Then it becomes an extriangulated category and its extriangulated structure $(\ch'_{\cq},\mathbb{E}',\mathfrak{s}')$ can be described as follows$\colon$
\begin{itemize}
	\item For any two objects $X,Y$ in $\ch'_{\cq}$, the $\mathbb{E}'$-extension space is given by
	$$\mathbb{E}'(X,Y)=\Hom_{\cc}(X,Y[1]).$$
	\item For any $\delta\in \mathbb{E}'(Z,X)=\Hom_{\cc}
	(Z,X[1])$, take a distinguished triangle
	$$X\xrightarrow{f} Y\xrightarrow{g} Z\xrightarrow{g} X[1]$$
	and define $\mathfrak{s}'(\delta)=[X\xrightarrow{f}Y\xrightarrow{g}Z]$. This $\mathfrak{\delta}$ does not depend on the choice of the distinguished
	triangle above.
\end{itemize}

Since $\cp$ is the subcategory of projective-injective objects of $\ch$, then $\cp$ is also projective-injective in $\ch'_{\cq}$. But $\ch'_{\cq}$ has more projective-injective objects. Denote by $\cp\cup\cq$ the full subcategory of $\ch'$ whose objects are the union $\obj(\cp)\cup\obj(\cq)$.

\begin{Prop}\label{Prop:H' is Fro}
	$(\ch'_{\cq},\mathbb{E}',\mathfrak{s}')$ is a Frobenius extriangulated category with projective-injective objects $\cp\cup\cq$.
\end{Prop}
\begin{proof}
It clear that $\cp$	is projective-injective in $\ch'_{\cq}$. Let $I$ be an object in $\cq$. For any distinguished triangle in $\ch'_{\cq}$
$$X\ra Y\ra Z\dasharrow,$$
the space $\mathbb{E}(Z,I)=\Hom_{\cc}(Z[-1],I)=0$. Hence we have the following exact sequence
$$\Hom_{\ch}(Y,I)\ra\Hom_{\ch}(X,I)\ra0.$$ Thus any object in $\cq$ is injective.  

Similarly, since $\mathbb{E}'(I,X)\cong\mathbb{E}(I,X)\cong\Hom_{\underline{\ch}}(I,X[1])\cong\Hom_{\underline{\ch}}(X,I[d-2])\cong\Hom_{\cc}(X,I[d-1])=0$, then we have the following exact sequence
$$\Hom_{\ch}(I,Y)\ra\Hom_{\ch}(I,X)\ra0.$$ Thus any object in $\cq$ is projective. This shows that the objects in $\mathcal{P}\cup\mathcal{Q}$ are projective-injective.

By an argument analogous to Theorem \ref{Thm:Higgs is Fro}, the extriangulated category $\ch'_{\cq}$ has enough projective-injective objects and the full subcategory of projective-injective objects is $\cp\cup\cq$.
\end{proof}

\bigskip

%
%
%
%

There is an induced extriangulated structure $(\frac{\ch'_{\cq}}{[\cq]},\overline{\mathbb{E}}',\overline{\mathfrak{s}}')$ on $\frac{\ch'_{\cq}}{[\cq]}$ which is given by
\begin{itemize}
	\item For any $X,Y\in\frac{\ch'_{\cq}}{[\cq]}$, $\overline{\mathbb{E}}'(X,Y)=\mathbb{E}(X,Y)$.
	\item For any $Z,X$ in $\frac{\ch'_{\cq}}{[\cq]}$ and $\delta\in\mathbb{E}'(Z,X)=\Hom_{\cc}(Z,X[1])$, let $\overline{\mathfrak{s}}'(\delta)$ be the class $[X\xrightarrow{\overline{f}}Y\xrightarrow{\overline{g}}Z]$, where $\mathfrak{s}'(\delta)=[X\xrightarrow{f}Y\xrightarrow{g}Z]$.
\end{itemize}

\begin{Thm}\label{Thm:CY reduction of Higgs}
We have an equivalence of Frobenius extriangulated categories
$$\frac{\ch'_{\cq}}{[\cq]}\iso\ch_{\cq}.$$
\end{Thm}

\begin{proof}
By Lemma \ref{Lem:ZQF=H'Q}, Lemma \ref{Lem:reduction of F} and Theorem \ref{Thm:F=H}, we have the following diagram of $k$-linear equivalences
\[
\begin{tikzcd}
\frac{\cz_{\cq}\cap\cf}{[\cq]}\arrow[r,"\simeq"]\arrow[d,"\simeq"]&\frac{\ch'_{\cq}}{[\cq]}\\
\cf_{\cq}\arrow[r,"\simeq"]&\ch_{\cq}.
\end{tikzcd}
\]

We obtain a $k$-linear equivalence
$$\Phi\colon\frac{\ch'_{\cq}}{[\cq]}\iso\ch_{\cq}.$$

The functor $\Phi$ can be described as follows$\colon$
\begin{itemize}
	\item[(1)] Let $X$ be an object of $\ch'_{\cq}$. The object $\Phi(X)$ is given as follows$\colon$By Theorem \ref{Thm:left t-structure}, we have $t$-structure $(\ct^{'\leqslant0},\ct^{'\geqslant0})$ on $\ct$ where 
$\ct^{'\leqslant0}=^{\perp_{\ct}}\!\!\cm[<0]\cap\thick(\cp)^{\perp_{\mathcal{T}}}\cap^{\perp_{\mathcal{T}}}\thick(\cp)$ and $\ct^{'\geqslant0}=\pi^{-1}(^{\perp_{\cu}}\!\cm[>0])$. This gives a triangle in $\ct$
$$X'\ra X\ra Z_{X}\ra Y[1]$$
with $X'\in^{\perp_{\ct}}\!\!\cm[<d]\cap\thick(\cp)^{\perp_{\mathcal{T}}}\cap^{\perp_{\mathcal{T}}}\thick(\cp)\subseteq\ct^{\fd}$ and $Z_{X}\in\pi^{-1}(^{\perp_{\cu}}\!\cm[>d])$. Since $(\ct^{\leqslant0},\ct^{\geqslant0})$ is also a $t$-structure on $\ct$, there exits a triangle
$$\sigma^{\leqslant0}Z_{X}\ra Z_{X}\ra\sigma^{\geqslant1}Z_{X}\ra\sigma^{\leqslant0}Z[1]$$
with $\sigma^{\leqslant0}Z_{X}\in\ct^{\leqslant0}=\pi^{-1}(\cu^{\leqslant0})$ and $\sigma^{\geqslant1}Z_{X}\in\ct^{\geqslant1}=\cm[\geqslant\!0]^{\perp_{\mathcal{T}}}\cap\thick(\cp)^{\perp_{\mathcal{T}}}\cap\,^{\perp_{\mathcal{T}}}\thick(\cp)\subseteq\ct^{\fd}$. Then $\Phi(X)$ is given by $\sigma^{\leqslant0}Z_{X}$.

\item[(2)] Let $f$ be a morphism in $\ch'_{\cq}$. We lift it to be a morphism in $\cz_{\cq}\cap\cf$. Then $\Phi(f)$ is given by its image in $\ch_{\cq}$.
\end{itemize}

By the properties of $t$-structure, we see that $\Phi$ preserves the extriangulated structures. Hence $\Phi$ is an extriangulated functor. Then it is enough to show that $$\Phi\colon\overline{\mathbb{E}}'(X,Y)=\Hom_{\cc}(X,Y[1])\iso\Hom_{\cv/\cv^{\fd}}(\Phi(X),\Phi(Y)[1])$$ holds for any $X,Y\in\ch'_{\cq}$.

Let $f\colon X\ra Y[1]$ be a morphism in $\cc$. Assume that the induced morphism $\Phi(f)\colon\sigma^{\leqslant0}Z_{X}\ra\sigma^{\leqslant0}Z_{Y}[1]$ is zero in $\cv/\cv^{\fd}$. Then $\Phi(f)$ factors through some object $W$ in $\cv^{\fd}=\ct^{\fd}\cap\thick(\cq)^{\perp_{\ct}}\subseteq\ct^{\fd}$, i.e. we have the following commutative diagram in $\cv$
\[
\begin{tikzcd}
	\sigma^{\leqslant0}Z_{X}\arrow[rr]\arrow[dr]&&\sigma^{\leqslant0}Z_{Y}[1]\\
	&W\arrow[ur]&.
\end{tikzcd}
\]
Notice that $W$ lies in $\thick(\cq)^{\perp_{\ct}}\cap^{\perp_{\ct}}\thick(\cq)$. We lift the above commutative diagram to be a commutative diagram in $\ct$. This shows that $f$ also factors $W$. Thus $f$ must be zero. Hence
$\Phi\colon\overline{\mathbb{E}}'(X,Y)=\Hom_{\cc}(X,Y[1])\ra\Hom_{\cv/\cv^{\fd}}(\Phi(X),\Phi(Y)[1])$ is an injection.

Next we show that it is surjective. Let $g\colon\colon\sigma^{\leqslant0}Z_{X}\ra\sigma^{\leqslant0}Z_{Y}[1]$ be a morphism in $\cv/\cv^{\fd}$. By Theorem \ref{Thm:quadruple reduction}, $(\cv,\cv^{\fd},\cm,\cp)$ is also a $(d+1)$-Calabi--Yau quadruple. Denote by $\pi_{\cv}\colon\cv\ra\cv/\thick(\cp)$ the canonical quotient functor.

By Lemma \ref{Lem:morohism in C}, it has a representation of the form $\sigma^{\leqslant0}Z_{X}\xleftarrow{s}V\xrightarrow{h}\sigma^{\leqslant0}Z_{Y}[1]$ in $\cv$ such that $\cone(s)\in\pi_{\cv}^{-1}((\cv/\thick(\cp))_{\leqslant0})\cap\cv^{\fd}$. There is a triangle in $\cv$
$$V\xrightarrow{s}\sigma^{\leqslant0}Z_{X}\xrightarrow{\alpha}\cone(s)\ra V[1].$$

Then we lift the morphism $\alpha$ to be a morphism $\alpha'\colon\sigma^{\leqslant0}Z_{X}\ra\cone(s)$ in $\cz_{\cq}\subseteq\ct$. We form the following triangle in $\ct$
$$V'\xrightarrow{s'}\sigma^{\leqslant0}Z_{X}\xrightarrow{\alpha'}\cone(s)\ra V'[1].$$

Applying the functor $\Hom_{\ct}(?,\sigma^{\leqslant0}Z_{Y}[1])$ to the above triangle, we see that $\Hom_{\ct}(V',\sigma^{\leqslant0}Z_{Y}[1])\cong\Hom_{\cv}(V,\sigma^{\leqslant0}Z_{Y}[1])$. Hence there exits a morphism $h'\colon V'\ra\sigma^{\leqslant0}Z_{Y}[1]$ such that its image in $\cv$ is $h\colon V\ra\sigma^{\leqslant0}Z_{Y}[1]$. We obtain the following roof
in $\ct$
$$\sigma^{\leqslant0}Z_{X}\xleftarrow{s'}V'\xrightarrow{h'}\sigma^{\leqslant0}Z_{Y}[1]$$
which represents a morphism in $\cc=\ct/\ct^{\fd}$. This shows that $\Phi\colon\overline{\mathbb{E}}'(X,Y)=\Hom_{\cc}(X,Y[1])\ra\Hom_{\cv/\cv^{\fd}}(\Phi(X),\Phi(Y)[1])$ is surjective. The assertion follows.

\end{proof}

\begin{Rem}
	When \( d = 2 \), the additive quotient \( \ch'_{\cq}/[\cq] \) corresponds exactly to a reduction for stably 2-Calabi--Yau Frobenius extriangulated categories with respect to a functorially finite rigid subcategory, as discussed in \cite[Theorem 1]{FMP2023}. When the category $\cp$ is the zero category, this corresponds to Iyama--Yoshino's Calabi--Yau reduction \cite{IyamaYoshino2008}. 
\end{Rem}

To summarize, the operations fit into the following commutative diagram.
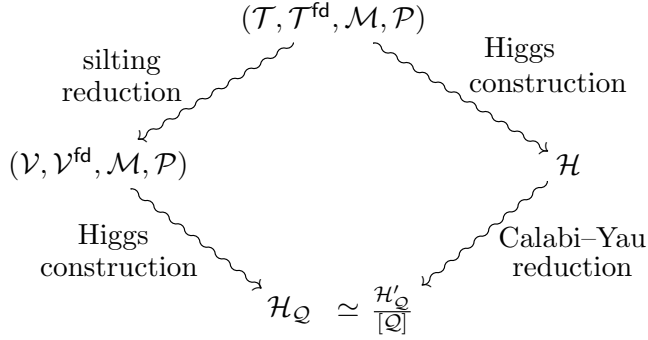
\[
\SelectTips{cm}{10}
\begin{xy}
	0;<1.1pt,0pt>:<0pt,-1.1pt>::
	(80,0) *+{(\ct,\ct^{\fd},\cm,\cp)} ="0", (0,50) *+{(\cv,\cv^{\fd},\cm,\cp)} ="1", (160,50) *+{\ch} ="2",
	(65,100) *+{\ch_{\cq}} ="3", (85,100) *+{\simeq} ="4", (100,100) *+{\frac{\ch'_{\cq}}{[\cq]}} ="5",
	(145,11) *+{\text{Higgs }}="6",
	(153,22) *+{\text{construction}}="7",
	(8,15) *+{\text{silting}}="8",
	(8,25) *+{\text{reduction}}="9",
	(5,75) *+{\text{Higgs}}="10",
	(7,85) *+{\text{construction}}="11",
	(163,75) *+{\text{Calabi--Yau }}="12",
	(163,85) *+{\text{reduction }}="13",
	(220,0) *+{},
	"0", {\ar@{~>}, "1"}, {\ar@{~>}, "2"},
	"1", {\ar@{~>}, "3"}, 
	"2", {\ar@{~>}, "5"},
\end{xy}
\]

\bigskip

\begin{Cor}\label{Cor:reduction in P}
	If $\cq$ is a full subcategory of $\cp$, then $\ch'_{\cq}=\ch$ and we have an equivalence of extriangulated categories
	$$\ch/[\cq]\simeq\ch_{\cq}.$$
\end{Cor}

\begin{Ex1}
	Let $(Q,F)$ be the following ice quiver
	
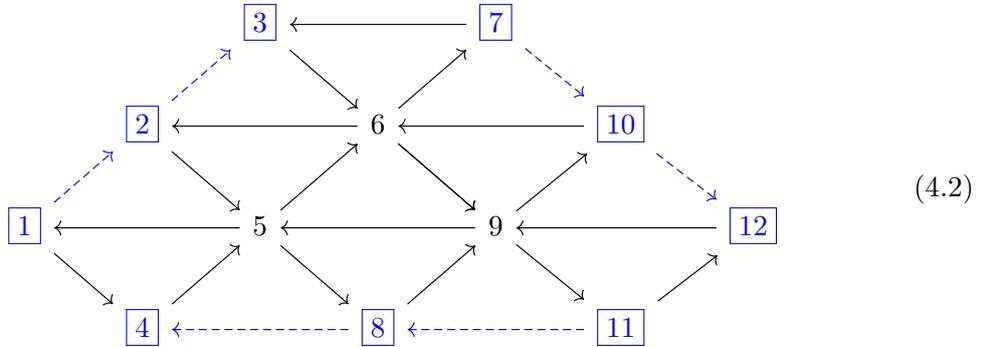
\begin{equation} \label{eq: quiver_basic triangle}
		\begin{tikzcd}[column sep=2em, row sep=3.5ex]
			& & \color{blue}\boxed{3} \arrow[rd]  && 
			\color{blue}\boxed{7} \arrow[rd, dashed,blue]  \arrow[ll]\\
			& \color{blue}\boxed{2} \arrow[ru, dashed,blue] \arrow[rd] && 
			6 \arrow[ll]  \arrow[rd] \arrow[ru] \arrow[rd] && 
			\color{blue}\boxed{10} \arrow[rd, dashed,blue] \arrow[ll] \\
			\color{blue}\boxed{1} \arrow[rd] \arrow[ru, dashed,blue]  && 
			5 \arrow[ru] \arrow[ll] \arrow[rd] && 
			9 \arrow[ru] \arrow[rd] \arrow[ll] & &
			\color{blue}\boxed{12} \arrow[ll] \\
			& \color{blue}\boxed{4} \arrow[ru] & & \color{blue}\boxed{8} \arrow[ru]  \arrow[ll, dashed,blue] & & \color{blue}\boxed{11}\arrow[ru]   \arrow[ll, dashed,blue] \\
		\end{tikzcd}
	\end{equation}
	where the frozen subquiver is the blue part. Let $W$ be the alternative sum of triangles. Then we get an ice quiver with potential $(Q,F,W)$.
	
	Denote by $\bm{\Gamma}$ the corresponding relative Ginzburg dg algebra. The ice quiver with potential $(Q,F,W)$ is Jacobi-finite, i.e. $H^{0}(\bm{\Gamma})$ is finite dimensional. Let $e$ be the sum of all idempotents associated with frozen vertices. Hence we obtain a $3$-Calabi--Yau quadruple $(\per(\bm{\Gamma}),\pvd_{e}(\bm{\Gamma}),\add(\bm{\Gamma}),\add(e\bm{\Gamma}))$.
	
	Denote by $\ch$ the associated Higgs category and $\cc=\per\bm{\Gamma}/\pvd_{e}(\bm{\Gamma})$ the associated relative cluster category. Let $e'=e_{4}+e_{8}+e_{11}+e_{7}+e_{10}+e_{12}$ and $\cq=\add(e'\bm{\Gamma})$. By the above Corollary, we have an equivalence of extriangulated categories
	$$\ch/[\cq]\simeq\ch_{\cq},$$
	where $\ch_{\cq}$ is the Higgs category associated with the following ice quiver $(Q',F')$
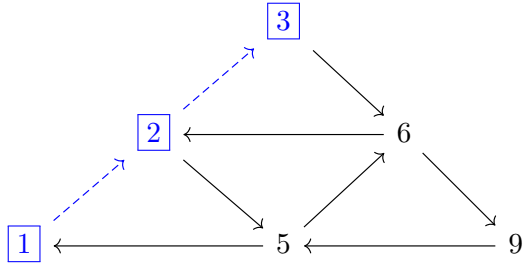
	\[
	\begin{tikzcd}
		&&\color{blue}\boxed{3}\ar[dr]&\\
		&\color{blue}\boxed{2}\ar[ur,dashed,blue]\ar[dr]&&6\ar[ll]\ar[dr]\\
		\color{blue}\boxed{1}\ar[ur,dashed,blue]&&5\ar[ll]\ar[ur]&&9\ar[ll]
	\end{tikzcd}
	\]
	with potential $W'$ given by the alternative sum of triangles. By \cite[Theorem 8.17]{Wu2023}, the Higgs category $\ch_{\cq}$ is equivalent to $\mathrm{mod}(\Pi(A_{3}))$, where $\Pi(A_{3})$ is the preprojective algebra of type $A_{3}$.
27
\end{Ex1}

\begin{Ex1}
Consider the following ice quiver $(Q,F)$
	\[
	\begin{tikzcd}
	&\color{blue}\boxed{1}\arrow[r,blue]\arrow[dddl,blue]&\color{blue}\boxed{2}\arrow[r,blue]\arrow[dl]&\color{blue}\boxed{3}\arrow[dl]\\
	&4\arrow[r]\arrow[u]&5\arrow[r]\arrow[dl]\arrow[u]&\color{blue}\boxed{6}\arrow[u,blue]\arrow[dl]\\
	&7\arrow[r]\arrow[u]&8\arrow[r]\arrow[u]&\color{blue}\boxed{9}\arrow[u,blue]\arrow[dlll,blue].\\
	\color{blue}\boxed{10}\arrow[ur]&&&
	\end{tikzcd}
	\]
Notice that this is the ice quiver which defines a cluster structure on Grassmannian $\mathrm{Gr}(3,6)$. Let $W$ be the sum of alternative triangles. Let $\bm{\Gamma}$ be the complete relative Ginzburg dg algebra of $(Q,F,W)$ and $e=e_{1}+e_{2}+e_{3}+e_{6}+e_{9}+e_{10}$. By \cite[Theorem 3.7]{PresslandPostnikov}, $\bm{\Gamma}$ is concentrated in degree 0 and $H^{0}(\bm{\Gamma})$ is Noetherian and $H^{0}(\bm{\Gamma})/\langle e\rangle$ is finite dimensional. This implies that $(\per(\bm{\Gamma}),\pvd_{e}(\bm{\Gamma}),\add(\bm{\Gamma}),\add(e\bm{\Gamma}))$ is a $3$-Calabi--Yau quadruple. Denote by $\cc,\ch$ the corresponding relative cluster category and Higgs category respectively. Let $B=eH^{0}(A)e$. By Corollary \ref{Cor:mainresult}, $\ch$ is equivalent to $\mathrm{gpr}(B)$.

Now let $\cq=\add(e_{10}\bm{\Gamma})$. The silting reduction of $(\per(\bm{\Gamma}),\pvd_{e}(\bm{\Gamma}),\add(\bm{\Gamma}),\add(e\bm{\Gamma}))$ with respect to $\cq$ is $(\per(\bm{\Gamma}'),\pvd_{f}(\bm{\Gamma}'),\add(\bm{\Gamma}'),\add(f\bm{\Gamma}'))$, where $\bm{\Gamma}'$ is the complete relative Ginzburg dg algebra of the following ice quiver $(Q',F')$
 \[
 \begin{tikzcd}
 	&\color{blue}\boxed{1}\arrow[r,blue]&\color{blue}\boxed{2}\arrow[r,blue]\arrow[dl]&\color{blue}\boxed{3}\arrow[dl]\\
 	&4\arrow[r]\arrow[u]&5\arrow[r]\arrow[dl]\arrow[u]&\color{blue}\boxed{6}\arrow[u,blue]\arrow[dl]\\
 	&7\arrow[r]\arrow[u]&8\arrow[r]\arrow[u]&\color{blue}\boxed{9}\arrow[u,blue]\\
 \end{tikzcd}
 \]
with potential $W'$ obtained from $W$ by deleting all cycles passing through vertex $\color{blue}\boxed{10}$ and $f=e_{1}+e_{2}+e_{3}+e_{6}+e_{9}$.

Let $\cc',\ch'$ be the associated relative cluster category and Higgs category respectively. By Corollary \ref{Cor:reduction in P}, the Calabi--Yau reduction $\frac{\ch}{[\cq]}$ of $\ch$ with respect to $\cq$ is the Higgs category $\ch'$. We have the canonical exact quotient functor between Frobenius extriangulated categories, see Theorem \ref{Thm:CY reduction of Higgs}
$$\pi\colon\ch\rightarrow\ch'\simeq\ch/[\cq].$$
Let $A$ be the preprojective algebra of type $A_{5}$. Notice that the relative Ginzburg dg algebra $\bm{\Gamma}'$ of $(Q',F',W')$ is also concentrated in degree 0 and $B'=fH^{0}(\bm{\Gamma}')f$ is the quotient of $A$ by an ideal, \cite[Proposition 8.6, Corollary 8.7]{Wu2023}. Hence $\ch'$ is equivalent to $\mathrm{gpr}(B')$ and $\mathrm{gpr}(B')$ is the subcategory $\mathrm{Sub}Q_{k}$ \cite{GLS2008} of $\mathrm{mod}(A)$ consisting of those modules with socle concentrated at a vertex $k$ \cite[Subsection 6.3]{KIY15}. Hence we obtain the following exact quotient functor between Frobenius extriangulated categories
$$\pi\colon\mathrm{gpr}(B)\ra\mathrm{Sub}Q_{k}.$$
Note that this quotient functor already appeared in the work of \cite{JKS2016}.

Let $\mathcal{R}=(e_{10}+e_{7})\bm{\Gamma}.$ The silting reduction of $(\per(\bm{\Gamma}),\pvd_{e}(\bm{\Gamma}),\add(\bm{\Gamma}),\add(e\bm{\Gamma}))$ with respect to $\mathcal{R}$ is $(\per(\bm{\Gamma}''),\pvd_{\alpha}(\bm{\Gamma}''),\add(\bm{\Gamma}''),\add(\alpha\bm{\Gamma}''))$, where $\bm{\Gamma}''$ is the complete relative Ginzburg dg algebra of the following ice quiver $(Q',F')$

\[
\begin{tikzcd}
	&\color{blue}\boxed{1}\arrow[r,blue]&\color{blue}\boxed{2}\arrow[r,blue]\arrow[dl]&\color{blue}\boxed{3}\arrow[dl]\\
	&4\arrow[r]\arrow[u]&5\arrow[r]\arrow[u]&\color{blue}\boxed{6}\arrow[u,blue]\arrow[dl]\\
	&&8\arrow[r]\arrow[u]&\color{blue}\boxed{9}\arrow[u,blue]\\
\end{tikzcd}
\]
with potential $W'$ obtained from $W$ by deleting all cycles passing through vertices $7$ and $\color{blue}\boxed{10}$ and  $\alpha=e_{1}+e_{2}+e_{3}+e_{6}+e_{9}$.

Let $\cc'',\ch''$ be the associated relative cluster category and Higgs category respectively. Let $$\ch'_{\mathcal{R}}=\{M\in\ch\,|\,\Hom_{\cc}(M,\mathcal{R}[1])=0\}\subseteq\ch.$$By Theorem \ref{Thm:CY reduction of Higgs}, the Calabi--Yau reduction $\frac{\ch'_{\mathcal{R}}}{[\mathcal{R}]}$ of $\ch$ with respect to $\mathcal{R}$ is the Higgs category $\ch''$.

%

\end{Ex1}

%
%

\bigskip

Now we assume that $\ct$ is an algebraic triangulated category, i.e. there exits a pretriangulated dg category $\ct_{dg}$ such that $H^{0}(\ct_{dg})\simeq\ct$. Then the categories $\ct^{\fd},\cm$ and $\cp$ are also admitted dg enhancements induced form $\ct$. We denote them by $\ct_{dg}^{\fd},\cm_{dg}$ and $\cp_{dg}$ respectively. Let $\cv_{dg}$ be the Drinfeld dg quotient $\ct_{dg}/\thick_{dg}(\cp)$ \cite{VD2004}.

\begin{Prop}\cite[Example 4.7]{Chen202402}
	Let $\ca$ be a pretriangulated dg category. Then we have
	\begin{itemize}
		\item[(1)] each morphism $f\colon A\ra B$ admits
		a homotopy cokernel
		\[
		\begin{tikzcd}
			A\ar[r,"f"] \ar[rr, bend right=8ex,"h"swap]&B\ar[r,"j"]&\cone(f).
		\end{tikzcd}
		\] Dually, each morphism admits a homotopy kernel.
		\item[(2)] a 3-term homotopy complex in $\ca$
		is homotopy left exact if and only if it is homotopy right exact.
		\item[(3)] the class of all homotopy short exact sequences in $\ca$ defines an exact structure on $\ca$.
	\end{itemize}
\end{Prop}

Then the Drinfeld dg quotient $\cc_{dg}=\ct_{dg}/\ct_{dg}^{\fd}$ is a dg enhancement of $\cc=\ct/\ct^{\fd}$. It carries the canonical exact dg structure, as described above.

Let $\ch'_{\cq,dg}$ be the full dg subcategory of $\cc_{dg}$ consisting of
objects in $\ch'_{\cq}$. It is an extension closed subcategory of $\cv_{dg}$. Then $\ch'_{\cq,dg}$ becomes a exact dg category in the sense of Xiaofa Chen \cite{Chen2306,Chen202402,Chen202406}. The conflations in $\ch'_{\cq,dg}$ are given by homotopy short exact sequences in $\cc_{dg}$ whose terms all lie in $\ch'_{\cq,dg}$ \cite[Example-Definition 4.8]{Chen202402}. 

Let $\cf_{dg}$ be the full dg subcategory of $\ct_{dg} $ consisting of objects in $\cf$ and $\cf_{\cq,dg}$ the the full dg subcategory of $\cv_{dg} $ of objects in $\cf_{\cq}$. Let $\cz_{\cq,dg}$ the full dg subcategory of $\ct_{dg} $ consisting of objects in $\cz_{\cq}$.

For a dg category $\ca$, denote by $\tau_{\leqslant0}\ca$  the dg category with the same objects as $\ca$ and
whose morphism complexes are given by $(\tau_{\leqslant0}\ca)(X,Y)=\tau_{\leqslant0}(\ca(X,Y))$, where $\tau_{\leqslant0}$ is the mild truncation functor.  

\begin{Lem}\cite[Lemma 3.30, Corollary 3.31]{Chen202402}
	The canonical dg functor $\ct_{dg}\ra\cc_{dg}$ induces a quasi-equivalence $\tau_{\leqslant0}\cf_{dg}\iso\tau_{\leqslant0}\ch_{dg}$.
\end{Lem} 
Then we have the following Corollary.
\begin{Cor}
	The canonical dg functor $\ct_{dg}\ra\cc_{dg}$ induces a quasi-equivalence
	$$\tau_{\leqslant0}(\cz_{\cq,dg}\cap\cf_{dg})\iso\tau_{\leqslant0}\ch'_{\cq,dg}.$$
\end{Cor}

Let $\cq_{dg}$ be the full dg subcategory of $\tau_{\leqslant0}\ch'_{\cq,dg}$ consisting of objects in $\cq$. By \cite[Theorem 3.23]{Chen202406}, the dg quotient $\frac{\tau_{\leqslant0}\ch'_{\cq,dg}}{\cq_{dg}}$ carries a canonical exact structure induced from $\tau_{\leqslant0}\ch'_{\cq,dg}$ and $H^{0}(\frac{\tau_{\leqslant0}\ch'_{\cq,dg}}{\cq_{dg}})=\frac{\ch'_{\cq}}{[\cq]}$.

\begin{Thm}\label{Thm:dg exact quasi-isomorphism of H'}
	We have an exact quasi-isomorphism of exact dg categories
	$$\frac{\tau_{\leqslant0}\ch'_{\cq,dg}}{\cq_{dg}}\simeq\tau_{\leqslant0}\ch_{\cq,dg}.$$
\end{Thm}
\begin{proof}
By \cite[pp.21]{Chen2306}, we have a quasi-isomorphism of pretriangulated dg categories
$$\frac{\tau_{\leqslant0}(\cz_{\cq,dg})}{\cq_{dg}}\iso\tau_{\leqslant0}\cv_{dg}.$$
It induces an exact quasi-isomorphism of dg exact categories
$$\frac{\tau_{\leqslant0}(\cz_{\cq,dg}\cap\cf_{dg})}{\cq_{dg}}\iso\tau_{\leqslant0}\cf_{dg}.$$

We obtain the following diagram of quasi-isomorphisms
\[
\begin{tikzcd}
	\frac{\tau_{\leqslant0}(\cz_{\cq,dg}\cap\cf_{dg})}{\cq_{dg}}\arrow[r,"\simeq"]\arrow[d,"\simeq"]&\frac{\tau_{\leqslant0}\ch'_{\cq}}{\cq_{dg}}\\
	\tau_{\leqslant0}\cf_{\cq,dg}\arrow[r,"\simeq"]&\tau_{\leqslant0}\ch_{\cq,dg},
\end{tikzcd}
\] 	
where the vertical quasi-isomorphism is also an exact quasi-isomorphism.

Hence $\frac{\tau_{\leqslant0}\ch'_{\cq}}{\cq_{dg}}$ is quasi-isomorphic to $\tau_{\leqslant0}\ch_{\cq,dg}$. One can check that they have the same exact dg structure and the assertion follows.

\end{proof}

Let $\cd^{b}_{dg}(\ch'_{\cq,dg})$ be the bounded dg derived category \cite[Theorem 3.1]{Chen202406} of $\ch'_{\cq,dg}$ with respect to the dg exact structure described as above. Let $\cd^{b}(\ch'_{\cq,dg})=H^{0}(\cd^{b}_{dg}(\ch'_{\cq,dg}))$. Denote by $\cd^{b}_{dg}(\ch'_{\cq,dg})/\thick_{dg}(\cq)\simeq\cc_{\cq,dg}$ the Drinfeld dg quotient of $\cd^{b}_{dg}(\ch'_{\cq,dg})$ by its full dg subcategory $\thick_{dg}(\cq)$.

Let $\cc_{\cq,dg}$ be the Drinfeld dg quotient $\cv_{dg}/\cv^{\fd}_{dg}$. By an analogue proof of \cite[ Proposition 3.32]{Chen202406}, the bounded dg derived category of $\ch_{\cq,dg}$ is $\cc_{\cq,dg}$.

\begin{Thm}\label{Thm:reduction of cc}
	We have a quasi-isomorphism of pretriangulated dg categories
	$$\cd^{b}_{dg}(\ch'_{\cq,dg})/\thick_{dg}(\cq)\simeq\cc_{\cq,dg}.$$
	In particular, we have an equivalence of triangulated categories
	$$\cd^{b}(\ch'_{\cq,dg})/\thick(\cq)\simeq\cc_{\cq}=\cv/\cv^{\fd}.$$
\end{Thm}

\begin{proof}
	It follows from Theorem \ref{Thm:dg exact quasi-isomorphism of H'} and \cite[Theorem 3.23]{Chen202406}.
\end{proof}
\begin{Rem}
	Note that $\ch'_{\cq,dg}$ is a full dg subcategory of $\ch_{dg}$. But they have different dg exact structures. It follows that, in general, the canonical triangle functor $\cd^{b}(\ch'_{\cq,dg})\ra\cd^{b}(\ch_{dg})$ is not fully faithful.
\end{Rem}
\bigskip

In summary, we obtain the following commutative diagram of operations.
\[
\SelectTips{cm}{10}
\begin{xy}
	0;<1.1pt,0pt>:<0pt,-1.1pt>::
	(80,0) *+{(\ct,\ct^{\fd},\cm,\cp)} ="0", (0,50) *+{(\cv,\cv^{\fd},\cm,\cp)} ="1", (160,50) *+{\cc=\ct/\ct^{\fd}} ="2",
	(65,100) *+{\cv/\cv^{\fd}} ="3", (85,100) *+{\simeq} ="4", (110,100) *+{\frac{\cd^{b}(\ch'_{\cq,dg})}{\thick(\cq)}} ="5",
	(145,11) *+{\text{Relative AGK's}}="6",
	(153,22) *+{\text{construction}}="7",
	(8,15) *+{\text{silting}}="8",
	(8,25) *+{\text{reduction}}="9",
	(-5,75) *+{\text{Relative AGK's}}="10",
	(-4,85) *+{\text{construction}}="11",
	(165,75) *+{\text{Calabi--Yau }}="12",
	(163,85) *+{\text{reduction}}="13",
	(220,0) *+{},
	"0", {\ar@{~>}, "1"}, {\ar@{~>}, "2"},
	"1", {\ar@{~>}, "3"}, 
	"2", {\ar@{~>}, "5"},
\end{xy}
\]

%
%
%
%
%
%

%

	\bibliographystyle{plain}

\end{document}